\documentclass[12pt,oneside,leqno]{amsart}
\usepackage{amssymb, enumerate, amsthm}
\usepackage{color}
\usepackage{graphicx}
\usepackage{amsmath}
\usepackage{tikz}
\usepackage{pgfplots}
\usepackage{tabu}
\usepackage[cmtip,all]{xy}

\usetikzlibrary{plotmarks}
\usetikzlibrary{intersections}
\usetikzlibrary{patterns}
\usepgfplotslibrary{fillbetween}
\usetikzlibrary{decorations.pathreplacing}

\newcommand{\U}{\mathbb{U}}
\newcommand{\N}{\mathbb{N}}
\newcommand{\M}{\mathbb{M}}

\newcommand{\define}{\stackrel{\rm def}{=}}
\newtheorem{theorem}{Theorem}
\newtheorem{lemma}[theorem]{Lemma}

\newtheorem{proposition}[theorem]{Proposition}
\newtheorem{remark}[theorem]{Remark}
\newtheorem{corollary}[theorem]{Corollary}
\newcounter{standard}
\newcounter{substandard}
\input{tcilatex}

\begin{document}
\title[The maximum dimension of a Lie nilpotent subalgebra]{The maximum
dimension of a Lie nilpotent subalgebra of $\boldsymbol{\mathbb{M}_n(F)}$ of
index $\boldsymbol{m}$}
\author{J. Szigeti}
\address{Institute of Mathematics, University of Miskolc, \hfill\break 3515
Miskolc-Egyetemv\'{a}ros, Hungary}
\email{matjeno@uni-miskolc.hu}
\author{J. van den Berg*}
\address{Department of Mathematics and Applied Mathematics, University of
Pretoria, Private Bag X20, Hatfield, Pretoria 0028, South Africa}
\email{john.vandenberg@up.ac.za}
\author{L. van Wyk}
\address{Department of Mathematical Sciences, Stellenbosch University,
\hfill\break Private Bag X1, Matieland~7602, Stellenbosch, South Africa}
\email{lvw@sun.ac.za}
\author{M. Ziembowski}
\address{Faculty of Mathematics and Information Science, \hfill\break
Technical University of Warsaw, 00-661 Warsaw, Poland}
\email{m.ziembowski@mini.pw.edu.pl}
\thanks{*Corresponding author}
\thanks{The first named author was partially supported by the National
Research, Development and Innovation Office of Hungary (NKFIH) K119934.}
\thanks{The second and third authors were supported by the National Research
Foundation of South Africa under Grant Numbers UID 85784 and UID 72375,
respectively. All opinions, findings and conclusions or recommendations
expressed in this publication are those of the authors and therefore the
National Research Foundation does not accept any liability in regard thereto.%
}
\thanks{The fourth author was supported by the Polish National Science
Centre grant UMO-2017/25/B/ST1/00384.}
\thanks{The authors thank K. C. Smith for fruitful discussions on the topic
of this paper.}
\keywords{Lie nilpotent, matrix algebra, Lie algebra, commutative
subalgebra, dimension}
\subjclass[2010]{Primary 16S50, 16U80; secondary 16R40}

\begin{abstract}
The main result of this paper is the following: if $F$ is any field and $R$
any $F$-subalgebra of the algebra $\mathbb{M}_n(F)$ of $n\times n$ matrices
over $F$ with Lie nilpotence index $m$, then 
\begin{equation*}
\mathrm{dim}_{F}R \leqslant M(m+1,n)
\end{equation*}
where $M(m+1,n)$ is the maximum of $\frac{1}{2}\!\left(n^{2}-%
\sum_{i=1}^{m+1}k_{i}^{2}\right)+1$ subject to the constraint $%
\sum_{i=1}^{m+1}k_{i}=n$ and $k_{1},k_{2},\ldots,k_{m+1}$ nonnegative
integers. This answers in the affirmative a conjecture by the first and
third authors. The case $m=1$ reduces to a classical theorem of Schur
(1905), later generalized by Jacobson (1944) to all fields, which asserts
that if $F$ is an algebraically closed field of characteristic zero, and $R$
is any commutative $F$-subalgebra of $\mathbb{M}_{n}(F)$, then $\mathrm{dim}%
_{F}R \leqslant \left\lfloor\frac{n^{2}}{4}\right\rfloor+1$. Examples
constructed from block upper triangular matrices show that the upper bound
of $M(m+1,n)$ cannot be lowered for any choice of $m$ and $n$. An explicit
formula for $M(m+1,n)$ is also derived.
\end{abstract}

\maketitle


\section*{Contents}

\begin{list}{{\arabic{standard}.}\hfill}{\usecounter{standard}\addtocounter{standard}{0}\setlength{\topsep}{1ex}\setlength{\labelwidth}{1.5\parindent}\setlength{\labelsep}{0.2\parindent}\setlength{\leftmargin}{2.3\parindent}\setlength{\listparindent}{0pt}\setlength{\itemsep}{5pt}}

\item Introduction

\item Preliminaries

\item The passage to local algebras over an algebraically closed field

\item Simultaneous triangularization and the passage to upper triangular matrix algebras

\item Subalgebras of $\U_{n}^{\ast}(F)$

\item Lie nilpotent subalgebras of $\U_{n}^{\ast}(F)$: the main theorem

\item The function $M(\ell,n)$

\item An illustrative example

\item Open questions

\end{list}


\section{Introduction}

In 1905 Schur \cite[Satz I, p.~67]{Sc05} proved that the dimension over the
field of complex numbers $\mathbb{C}$ of any commutative subalgebra of $%
\mathbb{M}_{n}(\mathbb{C})$ is at most $\left\lfloor\frac{n^{2}}{4}%
\right\rfloor+1$, where $\lfloor\hspace*{0.3em}\rfloor$ denotes the integer
floor function. Some forty years later, Jacobson \cite[Theorems 1 and 2,
p.~434]{Ja44} extended Schur's result by showing that the upper bound holds
for commutative subalgebras of $\mathbb{M}_{n}(F)$ for all fields $F$.

In a subsequent further improvement, Gustafson \cite[Section 2, p.~558]{Gu76}
showed that Schur's theorem in its most general form could be proved with
much greater efficiency using module theoretic methods. We record here that
Gustafson's elegant arguments are the inspiration for a key proposition in
this paper.

There have also appeared in the literature a number of papers offering
alternative proofs of Schur's theorem and its subsequent extensions. In this
regard, we refer the reader to \cite{ZG64}, \cite{Mi98} and \cite{Ka10}.

In response to a question posed in \cite[Section 5, Open problem (a), p.~562]%
{Gu76} Cowsik \cite{Co93} has proved a version of Schur's theorem for
artinian rings that are not algebras, in which the module length of a
faithful module substitutes for the dimension of the $F$-space on which the
matrices act.

The common approach to establishing Schur's upper bound has been to show
that if $F$ is a field and $R$ a commutative $F$-subalgebra of $\mathbb{M}%
_{n}(F)$, then there exist positive integers $k_{1}$ and $k_{2}$ such that $%
k_{1}+k_{2}=n$ and

\begin{equation*}
\mathrm{dim}_{F}R \leqslant k_{1}k_{2}+1.
\end{equation*}

\noindent An application of rudimentary Calculus then shows that

\begin{center}
$\mathrm{max}\hspace*{0.1em}\{k_{1}k_{2}+1:(k_{1},k_{2})\in\mathbb{N}\times%
\mathbb{N}$ and $k_{1}+k_{2}=n\}=\left\lfloor\frac{n^{2}}{4}\right\rfloor+1$,
\end{center}

\noindent whence $\mathrm{dim}_{F}R \leqslant \left\lfloor\frac{n^{2}}{4}%
\right\rfloor+1$.

The upper bound of \raisebox{0pt}[1.3\baselineskip][0.7\baselineskip]{}$%
\left\lfloor\frac{n^{2}}{4}\right\rfloor+1$ is, moreover, easily seen to be
optimal. Indeed, let $F$ be any field and $(k_{1},k_{2})$ any pair of
positive integers satisfying $k_{1}+k_{2}=n$. Define rectangular array $B$ by

\begin{equation*}
B\,\overset{\mathrm{def}}{=}\{(i,j)\in\mathbb{N}\times\mathbb{N}:1\leqslant
i\leqslant k_{1}< j\leqslant n\},
\end{equation*}

\noindent and subset $J$ of $\mathbb{M}_{n}(F)$ by

\begin{align}
J\,\overset{\mathrm{def}}{=}\left\{ 
\raisebox{0.7ex}[3.2ex][0ex]{
$\displaystyle\sum\limits_{(i,j)\in B}b_{ij}E_{(i,j)}\::\:b_{ij}\in F \hspace*{0.6em}\forall(i,j)\in B$}%
\right\},
\end{align}

\noindent where $E_{(i,j)}$ denotes the matrix unit in $\mathbb{M}_{n}(F)$
associated with position~$(i,j)$. Observe that $J$ comprises the set of all
block upper triangular matrices that correspond with $B$; it has the
following illuminating pictorial representation (the unshaded region in the
picture below corresponds with zero entries):

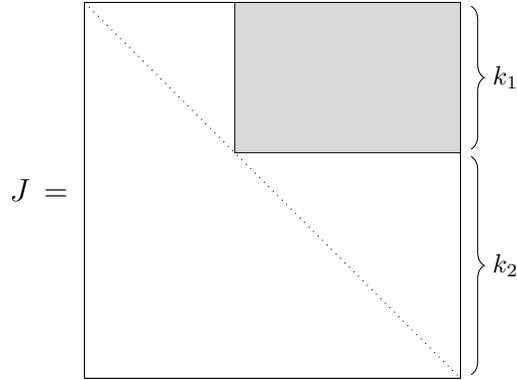
\vspace*{1\baselineskip}

\begin{center}
\begin{tikzpicture}
\draw {[line width=0.1pt](0,0) rectangle (5,5)};
\draw [dotted](0,5)--(5,0);
\draw [fill=gray!30][line width=0.1pt](5,5) rectangle (2,3);

\node [left] at (0,2.5) {$J\,=\,$};

\draw [decorate,decoration={brace,amplitude=5pt,mirror,raise=4pt},yshift=0pt]
(5,3.05) -- (5,4.95) node [black,midway,xshift=0.6cm]{\footnotesize $k_{1}$};

\draw [decorate,decoration={brace,amplitude=5pt,mirror,raise=4pt},yshift=0pt]
(5,0.05) -- (5,2.95) node [black,midway,xshift=0.6cm]{\footnotesize $k_{2}$};

\end{tikzpicture}
\end{center}

\vspace*{1\baselineskip}

\noindent Denote by

\begin{equation*}
FI_{n}\,\overset{\mathrm{def}}{=}\{aI_{n}:a\in F\} \,%
\makebox{ {\rm ($I_{n}$
is the $n\times n$ identity matrix)}}
\end{equation*}

\noindent the set of all $n\times n$ scalar matrices over $F$, and define

\begin{align}
R\,\overset{\mathrm{def}}{=} FI_{n}+J.
\end{align}

\noindent It is easily seen that $R$ is a local $F$-subalgebra of $\mathbb{M}%
_{n}(F)$ with (Jacobson) radical $J(R)=J$ such that $J^{2}=0$. This entails $%
R$ is commutative. It is clear too, that

\begin{equation*}
\mathrm{dim}_{F}R=k_{1}k_{2}+1.
\end{equation*}

\noindent The above simple construction shows that the upper bound $%
\left\lfloor\frac{n^{2}}{4}\right\rfloor+1=\mathrm{max}\hspace*{0.1em}%
\{k_{1}k_{2}+1:(k_{1},k_{2})\in\mathbb{N}\times\mathbb{N}$ and $%
k_{1}+k_{2}=n\}$ cannot be lowered for any $n\geqslant 2$, and is thus
optimal, as claimed.

We construct now an $F$-subalgebra $R$ of $\mathbb{M}_{n}(F)$ similar to the
one constructed above, but whose radical $J$ comprises $m$ blocks rather
than a single block. We require first a compact notation for the description
of such rings. To this end, let $k_{1},k_{2},\ldots,k_{m+1}$ be a sequence
of positive integers such that $k_{1}+k_{2}+\cdots+k_{m+1}=n$. For each $%
p\in\{1,2,\ldots,m\}$, define rectangular array

\begin{equation*}
B_{p}\,\overset{\mathrm{def}}{=} 
\begin{cases}
\{(i,j)\in\mathbb{N}\times\mathbb{N}:1\leqslant i\leqslant k_{1}<j\leqslant
n\},\makebox{ {\rm if} }p=1, \\[2ex] 
\{(i,j)\in\mathbb{N}\times\mathbb{N}:%
\parbox[t]{20em}{$k_{1}+k_{2}+\cdots+k_{p-1}<i\leqslant k_{1}+k_{2}+\cdots+k_{p}$\\[0.2\baselineskip]
$<j\leqslant n\}$, if $p>1$.}%
\end{cases}%
\end{equation*}

\noindent Put

\begin{align}
B\,\overset{\mathrm{def}}{=}\bigcup\limits_{p=1}^{m}B_{p}.
\end{align}

\noindent Define $J$ as in (1) but with $B$ defined as in (3) above. The
following pictorial representation of $J$ reveals a stack of $m$ blocks

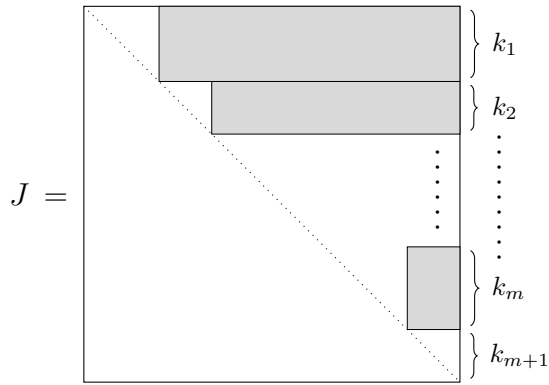
\vspace*{1\baselineskip}

\begin{center}
\begin{tikzpicture}

\draw [dotted](0,5)--(5,0);

\draw {[line width=0.1pt](0,0) rectangle (5,5)};

\draw [fill=gray!30][line width=0.1pt](5,5) rectangle (1,4);
\draw [fill=gray!30][line width=0.1pt](5,4) rectangle (1.7,3.3);
\draw [fill=gray!30][line width=0.1pt](5,1.8) rectangle (4.3,0.7);

\node [right] at (4.5,2.05) {$\cdot$};
\node [right] at (4.5,2.25) {$\cdot$};
\node [right] at (4.5,2.45) {$\cdot$};
\node [right] at (4.5,2.65) {$\cdot$};
\node [right] at (4.5,2.85) {$\cdot$};
\node [right] at (4.5,3.05) {$\cdot$};

\node [right] at (5.32,1.65) {$\cdot$};
\node [right] at (5.32,1.85) {$\cdot$};
\node [right] at (5.32,2.05) {$\cdot$};
\node [right] at (5.32,2.25) {$\cdot$};
\node [right] at (5.32,2.45) {$\cdot$};
\node [right] at (5.32,2.65) {$\cdot$};
\node [right] at (5.32,2.85) {$\cdot$};
\node [right] at (5.32,3.05) {$\cdot$};
\node [right] at (5.32,3.25) {$\cdot$};

\node [left] at (0,2.5) {$J\,=\,$};

\draw [decorate,decoration={brace,amplitude=3pt,mirror,raise=4pt},yshift=0pt]
(5,4.05) -- (5,4.95) node [black,midway,xshift=0.6cm]{\footnotesize $k_{1}$};

\draw [decorate,decoration={brace,amplitude=2.5pt,mirror,raise=4pt},yshift=0pt]
(5,3.35) -- (5,3.95) node [black,midway,xshift=0.6cm]{\footnotesize $k_{2}$};

\draw [decorate,decoration={brace,amplitude=3.1pt,mirror,raise=4pt},yshift=0pt]
(5,0.75) -- (5,1.75) node [black,midway,xshift=0.65cm]{\footnotesize $k_{m}$};

\draw [decorate,decoration={brace,amplitude=2.5pt,mirror,raise=4pt},yshift=0pt]
(5,0.05) -- (5,0.65) node [black,midway,xshift=0.8cm]{\footnotesize $k_{m+1}$};

\end{tikzpicture}
\end{center}

\vspace*{1\baselineskip}

\noindent We shall call the $F$-algebra $R$ defined as in (2), the \emph{%
algebra of $n\times n$ matrices over $F$ of type $(k_{1},k_{2},%
\ldots,k_{m+1})$.} We see that $R$ is again a local $F$-subalgebra of $%
\mathbb{M}_{n}(F)$ with radical $J(R)=J$ such that $J^{m+1}=0$ and

\begin{align}
\mathrm{dim}_{F}R &=k_{1}(n-k_{1})+k_{2}(n-k_{1}-k_{2})+\cdots  \notag \\
& \hspace*{1.1em}+k_{m}(n-k_{1}-k_{2}-\cdots -k_{m})+1.  \notag \\
&= \sum_{j=1}^{m}k_{j}\left(n-\sum_{i=1}^{j}k_{i}\right)+1.
\end{align}

\noindent A routine inductive argument shows that the expression (less $1$)
appearing on the right-hand-side of (4), simplifies as

\begin{equation*}
\sum_{j=1}^{m}k_{j}\left(n-\sum_{i=1}^{j}k_{i}\right)=\frac{1}{2}%
\left(n^{2}-\sum_{i=1}^{m+1}k_{i}^{2}\right)=\sum_{i,j=1,%
\,i<j}^{m+1}k_{i}k_{j},
\end{equation*}

\noindent so that (4) becomes

\begin{align}
\mathrm{dim}_{F}R=\frac{1}{2}\left(n^{2}-\sum_{i=1}^{m+1}k_{i}^{2}\right)+1=%
\sum_{i,j=1,\,i<j}^{m+1}k_{i}k_{j}+1.
\end{align}

The algebra of $n\times n$ matrices over $F$ of type $(k_{1},k_{2},%
\ldots,k_{m+1})$ is clearly not commutative (unless $m=1$), but it does
satisfy a weak form of commutativity called Lie nilpotence. To put this
notion in context, we first recall some basic facts about Lie algebras.

Let $\mathfrak{g}$ be a Lie algebra\footnote{%
Our Lie algebras are over a commutative ring that is not necessarily a
field. No harm shall come of this more general interpretation since, in the
few instances where results about standard Lie algebras are used, the
underlying commutative ring is a field.} and $x_{1},x_{2},\ldots,x_{m}$ a
finite sequence of elements in $\mathfrak{g}$. We define element $%
[x_{1},x_{2},\ldots,x_{m}]^{\ast}$ of $\mathfrak{g}$ recursively as follows

\begin{align*}
[x_{1}]^{\ast} &\overset{\mathrm{def}}{=} x_{1},\hspace*{0.3em}%
\makebox{\rm
and} \\[0.5ex]
[x_{1},x_{2},\ldots,x_{m}]^{\ast} &\overset{\mathrm{def}}{=}
[[x_{1},x_{2},\ldots,x_{m-1}]^{\ast},x_{m}],\hspace*{0.3em}%
\makebox{\rm for
$m>1$.}
\end{align*}

\noindent Recall that if ${\mathfrak{h}}$ is any ideal of ${\mathfrak{g}}$,
then the \emph{Lower Central Series} $\{{\mathfrak{h}}_{[m]}\}_{m\in\mathbb{N%
}}$ of ${\mathfrak{h}}$ is defined by

\begin{equation*}
{\mathfrak{h}}_{[m]}\overset{\mathrm{def}}{=}\{[x_{1},x_{2},\ldots,x_{m}]^{%
\ast}: x_{i}\in{\mathfrak{h}} \,%
\makebox{ {\rm for $1\leqslant i\leqslant
m$}}\}.
\end{equation*}

\noindent We say ${\mathfrak{g}}$ is \emph{nilpotent} if ${\mathfrak{g}}%
_{[m]}=0$ for some $m\in\mathbb{N}$, $m>1$, and more specifically, \emph{%
nilpotent of index $m$}, if ${\mathfrak{g}}_{[m+1]}=0$.

Every ring $R$ may be endowed with the structure of a Lie algebra (over the
centre of $R$), by choosing as bracket the \emph{commutator} defined by

\begin{equation*}
\forall r,s\in R, \,[r,s]\,\overset{\mathrm{def}}{=} rs-sr.
\end{equation*}

\noindent Following \cite[p. 4785]{SvW15}, we call a ring $R$ \emph{Lie
nilpotent} [resp.~\emph{Lie nilpotent of index $m$}] if $R$, considered as a
Lie algebra via the commutator, is nilpotent [resp.~nilpotent of index $m$].
The reader will observe that the commutative rings are precisely the rings
that are Lie nilpotent of index $1$.

A ring $R$ is said to satisfy the \emph{Engel condition of index $m$} if the
identity

\begin{equation*}
[x,\overbrace{y,\ldots,y}^{m\mathrm{\ times}}]^{\ast}=0,
\end{equation*}

\noindent holds in $R$. A ring is said to satisfy the \emph{Engel condition}
if it satisfies the Engel condition of index $m$ for some $m\in{\mathbb{N}}$%
. Clearly a ring that is Lie nilpotent of index $m$ satisfies the Engel
condition of index $m$. The following result of Riley and Wilson \cite[p. 974%
]{RW99} establishes a partial converse.


\begin{proposition}
If $F$ is any field and $R$ an $F$-algebra that is generated by a finite
number $d$ of elements, and $R$ satisfies the Engel condition of index $m$,
then $R$ is Lie nilpotent of index $f(d,m)\geqslant m$, where the index $%
f(d,m)$ depends only on $d$ and $m$.
\end{proposition}

Lie nilpotent rings have been shown to play an important role in the proofs
of certain classical results about polynomial and trace identities in the $F$%
-algebra $\mathbb{M}_{n}(F)$ (see \cite{Dr00} and \cite{DrFu04}). For fields 
$F$ of characteristic zero, Kemer's \cite{Ke91} pioneering work on the
T-ideals of associative algebras has revealed the importance of identities
satisfied by $n\times n$ matrices over the Grassmann (exterior) algebra $%
E=F\langle\{x_{i}:i\in{\mathbb{N}}\}:x_{i}x_{j}+x_{j}x_{i}=0$ whenever $%
1\leqslant i\leqslant j\rangle$ generated by an infinite family $\{x_{i}:i\in%
{\mathbb{N}}\}$ of anticommutative indeterminates. For $n\times n$ matrices
over a Lie nilpotent ring of index $m$, a Cayley-Hamilton identity of degree 
$n^{m}$ (with left- or right-sided scalar coefficients) was found in \cite%
{Sz97}. Since the Grassmann algebra $E$ is Lie nilpotent of index $m=2$, the
aforementioned Cayley-Hamilton identity for matrices in $\mathbb{M}_{n}(E)$
is of degree $n^{2}$. In \cite{Do98}, Domokos presents a slightly modified
version of this identity in which the coefficients are invariant under the
conjugation action of $\mathrm{GL}_{n}(F)$.

This paper is an attempt to answer a conjecture posed in~\cite[p.~4785]%
{SvW15}. The statement of this conjecture is rendered less cumbersome if
expressed in terms of a function $M(\ell,n)$ of positive integer arguments $%
\ell$ and $n$, defined as follows

\begin{align}
M(\ell,n)\,\overset{\mathrm{def}}{=} & \hspace*{0.4em}\mathrm{max}\left\{%
\frac{1}{2}\left(n^{2}-\sum_{i=1}^{\ell}k_{i}^{2}\right)+1:k_{1},k_{2},%
\ldots,k_{\ell} \makebox{ {\rm are}}\right.  \notag \\
& \left.\,\makebox{{\rm nonnegative integers such that} }\sum_{i=1}^{%
\ell}k_{i}=n\right\}.
\end{align}

\begin{quote}
\textbf{Conjecture.} \emph{Let $F$ be any field, $m$ and $n$ positive
integers, and $R$ an $F$-subalgebra of $\mathbb{M}_{n}(F)$ with Lie
nilpotence index $m$. Then }
\end{quote}

\emph{%
\begin{align}
\mathrm{dim}_{F}R \leqslant M(m+1,n).
\end{align}
}

\medskip

\noindent We shall henceforth refer to the above as \emph{`the Conjecture'}.
More specifically, if $F$ is any fixed field, we shall say that \emph{`the
Conjecture holds in respect of $F$'}, if (7) holds for all positive integers 
$m$ and $n$, and $F$-subalgebras $R$ of $\mathbb{M}_{n}(F)$ with Lie
nilpotence index $m$.

If $R$ is any algebra over a field $F$, then a module $V$ over $R$ is
precisely a representation of $R$ via action on the underlying $F$-space
structure on $V$. If the module is faithful, then this representation is
faithful thus yielding an embedding of $R$ into $\mathrm{End}_{F}V$, the $F$%
-algebra of $F$-space endomorphisms on $V$. If $V$ is also finite
dimensional over $F$, say $\mathrm{dim}_{F}V=n$, then $\mathrm{End}_{F}V$ is
isomorphic to $\mathbb{M}_{n}(F)$ and so we have an $F$-algebra embedding of 
$R$ into $\mathbb{M}_{n}(F)$. (We point out that such a finite dimensional $%
V $ is certain to exist if $R$ is finite dimensional, for $V$ can always be
chosen to be $R$ itself.) Thus, seen through a representation theoretic
lens, inequality (7) sheds light on a possible \emph{lower} bound for the
dimension of a faithful module over a given Lie nilpotent algebra.

In the same spirit, Domokos \cite[Theorem 1, p. 156]{Do03} derives a lower
bound for the dimension of a faithful module over a finite dimension algebra
satisfying the polynomial identity $[x_{1},y_{1}][x_{2},y_{2}]\ldots[%
x_{m},y_{m}]=0$, in terms of $m$.

Our initial task, which is easily accomplished, shall be to argue that the
upper bound (7) is optimal for all choices of $m$ and $n$.

Suppose first that $m+1\leqslant n$. It is proven in Corollary 27(a) that
for such $m$ and $n$, $M(m+1,n)=\frac{1}{2}\left(n^{2}-%
\sum_{i=1}^{m+1}k_{i}^{2}\right)+1$ for some sequence of \emph{positive}
integers $k_{1},k_{2},\ldots,k_{m+1}$ satisfying $\sum_{i=1}^{m+1}k_{i}=n$.
Let $F$ be any field and $R$ the algebra of $n\times n$ matrices over $F$ of
type $(k_{1},k_{2},\ldots,k_{m+1})$. As noted earlier, $R$ has the form $%
R=FI_{n}+J$ with radical $J$ satisfying $J^{m+1}=0$. Since the set $FI_{n}$
of scalar matrices is central in $R$, it can be shown that the $k$th terms
of the Lower Central Series for $R$ (interpreted as a Lie algebra via the
commutator) and $J$ coincide, that is to say, $R_{[k]}=J_{[k]}$, for $k>1$.
It is also evident that $J_{[k]}\subseteq J^{k}$ for every $k\in\mathbb{N}$.
Thus $R_{[m+1]}=J_{[m+1]}\subseteq J^{m+1}=0$, so $R$ is Lie nilpotent of
index $m$. It follows from (5) that $\mathrm{dim}_{F}R=\frac{1}{2}\left[%
n^{2}-\sum_{i=1}^{m+1}k_{i}^{2}\right]+1=M(m+1,n)$.

Now suppose $m+1>n$. No generality is lost if we suppose $n>1$. It is proven
in Corollary 27(b) that for such $m$ and $n$, $M(m+1,n)=M(n,n)=\frac{1}{2}%
(n^{2}-n)+1$, and this, by (5), is equal to $\mathrm{dim}_{F}R$ where $R$ is
the algebra of $n\times n$ matrices over field $F$ of type $%
(k_{1},k_{2},\ldots,k_{n})$ with $k_{1}=k_{2}=\cdots=k_{n}=1$. (The reader
will see that in this instance, $R$ is just the algebra of all upper
triangular matrices over $F$ with constant main diagonal.) As shown in the
previous paragraph, such an algebra $R$ is Lie nilpotent of index $n-1$ and
thus Lie nilpotent of index $m$, since $m\geqslant n-1$.

The theorem below collects together the conclusions drawn above.


\begin{theorem}
Let $F$ be any field, and $m$ and $n$ arbitrary positive integers. Then
there exists an $F$-subalgebra $R$ of $\mathbb{M}_{n}(F)$ with Lie
nilpotence index $m$ such that

\begin{equation*}
\mathrm{dim}_{F}R=M(m+1,n).
\end{equation*}
\end{theorem}

The main body of theory in this paper is developed in Sections 5 and~6 with
module theoretic methods our primary tools. Sections 3 and~4 show that the
Conjecture reduces to a consideration of local subalgebras of upper
triangular matrix rings over an algebraically closed field. Section 7, which
can be read independently of earlier sections, establishes important
properties of the function $M(\ell,n)$ required in earlier theory. An
explicit formula for $M(\ell,n)$ is also derived which is then shown to have
a more simplified form for small values of $\ell$. In Section~8 the algebra
of $n\times n$ matrices of type $(d_{1},d_{2},\ldots,d_{\ell})$ is used to
provide a pictorial representation of the objects introduced in earlier
theory. The content of Section 9, which is titled Open questions, is
self-evident.


\section{Preliminaries}

The symbol $\subseteq$ denotes containment and $\subset$ proper containment
for sets.

If $X$ is any set, then $X^{n}$ denotes the cartesian product of $n$ copies
of $X$.

$\mathbb{N}$ and $\mathbb{N}_{0}$ will denote the sets of positive integers,
and nonnegative integers, respectively.

All rings are associative and possess identity, and all modules are unital.

Let $R$ be a ring and $V$ a right $R$-module. We write $W\leqslant V$ to
indicate that $W$ is a submodule of $V$. If $X$ is a nonempty subset of $V$
and $I$ is a right ideal of $R$, then

\begin{equation*}
(0:^{I}X)\,\overset{\mathrm{def}}{=}\{a\in I:Xa=0\}=I\cap(0:^{R}X).
\end{equation*}

\noindent Observe that $(0:^{I}X)$ is always a right ideal of $R$.

Let $F$ be a field. For each $n\in\mathbb{N}$, \,$\mathbb{M}_{n}(F)$ [resp.~$%
\mathbb{U}_{n}(F)$] [resp.~$\mathbb{U}_{n}^{\ast}(F)$] shall denote the $F$%
-algebra of all $n\times n$ matrices over $F$ [resp.~upper triangular $%
n\times n$ matrices over $F$] [resp.~upper triangular $n\times n$ matrices
over $F$ with constant main diagonal].


\section{The passage to local algebras over an algebraically closed field}

In this section we show that the Conjecture reduces to a consideration of
local algebras over an algebraically closed field.


\begin{lemma}
Let $F$ be a subfield of field $K$ and $R$ an $F$-algebra. Let $r_{1}\otimes
b_{1},r_{2}\otimes b_{2},\ldots,r_{m}\otimes b_{m}\in R\otimes_{F}K$ with $%
r_{i}\in R$ and $b_{i}\in K$ for $i\in\{1,2,\ldots,m\}$. Then

\begin{equation*}
[r_{1}\otimes b_{1},r_{2}\otimes b_{2},\ldots,r_{m}\otimes
b_{m}]^{\ast}=[r_{1},r_{2},\ldots,r_{m}]^{\ast}\otimes(b_{1}b_{2}\ldots
b_{m}).
\end{equation*}
\end{lemma}

\begin{proof}
We provide only a proof of the inductive step.

Putting $r=[r_{1},r_{2},\ldots,r_{m}]^{\ast}$ and $b=b_{1}b_{2}\ldots b_{m}$
we see that

\begin{align*}
&[r_{1}\otimes b_{1},r_{2}\otimes b_{2},\ldots,r_{m+1}\otimes b_{m+1}]^{\ast}
\\[0.5ex]
&=[r\otimes b,r_{m+1}\otimes b_{m+1}]\hspace*{0.6em}%
\makebox{\rm [by the
inductive hypothesis]} \\[0.5ex]
&=(r\otimes b)(r_{m+1}\otimes b_{m+1})-(r_{m+1}\otimes b_{m+1})(r\otimes b)
\\[0.5ex]
&=(rr_{m+1})\otimes (bb_{m+1})-(r_{m+1}r)\otimes (b_{m+1}b) \\[0.5ex]
&=%
\parbox[t]{26em}{$(rr_{m+1})\otimes (bb_{m+1})-(r_{m+1}r)\otimes (bb_{m+1})$\hspace*{0.6em}[because $K$ is a field\\[0.5ex]
\,so $b_{m+1}b=bb_{m+1}$]} \\[0.5ex]
&=(rr_{m+1}-r_{m+1}r)\otimes(bb_{m+1}) \\[0.5ex]
&=[r_{1},r_{2},\ldots,r_{m+1}]^{\ast}\otimes(b_{1}b_{2}\ldots b_{m+1}).
\end{align*}
\end{proof}


\begin{proposition}
Let $F$ be a subfield of field $K$ and $R$ an $F$-subalgebra of $\mathbb{M}%
_{n}(F)$. Then:

\begin{list}{{(\alph{standard})}\hfill}{\usecounter{standard}\addtocounter{standard}{0}\setlength{\topsep}{1ex}\setlength{\labelwidth}{1.5\parindent}\setlength{\labelsep}{0.2\parindent}\setlength{\leftmargin}{2.3\parindent}\setlength{\listparindent}{0pt}\setlength{\itemsep}{5pt}}

\item ${\rm dim}_{F}R=\,{\rm dim}_{K}(R\otimes_{F}K)$.

\item $R\otimes_{F}K$ is isomorphic to a $K$-subalgebra of $\M_{n}(K)$.

\item If $R$ is Lie nilpotent of index $m$, then so is $R\otimes_{F}K$.

\end{list}
\end{proposition}

\begin{proof}
(a) is standard theory - see for example \cite[Exercise 19.3, p.~231]{AF92}.

(b) The hypothesis entails $R\otimes_{F}K$ is a $K$-subalgebra of $\mathbb{M}%
_{n}(F)\otimes_{F}K$. The result follows noting that $\mathbb{M}%
_{n}(F)\otimes_{F}K\cong\mathbb{M}_{n}(K)$ as $K$-algebras (see \cite[%
Chapter 9, Exercise 10, p. 94]{Pa91}).

(c) Suppose $R$ is Lie nilpotent of index $m$. Take $x_{1},x_{2},%
\ldots,x_{m+1}\in R\otimes_{F}K$. Since the expression $[x_{1},x_{2},%
\ldots,x_{m+1}]^{\ast}$ is additive in each of its $m+1$ arguments, $%
[x_{1},x_{2},\ldots,x_{m+1}]^{\ast}$ is expressible as a sum of elements of
the form $[r_{1}\otimes b_{1},r_{2}\otimes b_{2},\ldots,r_{m+1}\otimes
b_{m+1}]^{\ast}$ where $r_{i}\in R$ and $b_{i}\in K$ for $%
i\in\{1,2,\ldots,m+1\}$. By Lemma 3

\begin{align*}
&[r_{1}\otimes b_{1},r_{2}\otimes b_{2},\ldots,r_{m+1}\otimes b_{m+1}]^{\ast}
\\[0.5ex]
&=[r_{1},r_{2},\ldots,r_{m+1}]^{\ast}\otimes(b_{1}b_{2}\ldots b_{m+1}) \\%
[0.5ex]
&=0\otimes(b_{1}b_{2}\ldots b_{m+1})\hspace*{0.6em}%
\makebox{\rm [because $R$
is Lie nilpotent of index $m$]} \\[0.5ex]
&=0.
\end{align*}

\noindent It follows that $[x_{1},x_{2},\ldots,x_{m+1}]^{\ast}=0$, so $%
R\otimes_{F}K$ is Lie nilpotent of index $m$.
\end{proof}


\begin{theorem}
Let ${\mathcal{C}}$ be a nonempty class of fields and $\overline{\mathcal{C}}
$ the class of all subfields of fields in ${\mathcal{C}}$. The following
statements are equivalent:

\begin{list}{{(\alph{standard})}\hfill}{\usecounter{standard}\addtocounter{standard}{0}\setlength{\topsep}{1ex}\setlength{\labelwidth}{1.1\parindent}\setlength{\labelsep}{0.2\parindent}\setlength{\leftmargin}{2.3\parindent}\setlength{\listparindent}{0pt}\setlength{\itemsep}{5pt}}

\item The Conjecture holds in respect of all fields in ${\mathcal C}$;

\item The Conjecture holds in respect of all fields in $\overline{\mathcal C}$.
\end{list}
\end{theorem}

\begin{proof}
(b)$\Rightarrow$(a) is obvious since ${\mathcal{C}}\subseteq\overline{%
\mathcal{C}}$.

(a)$\Rightarrow$(b) Let $m$ and $n$ be positive integers, $F\in\overline{%
\mathcal{C}}$, and $R$ an $F$-subalgebra of $\mathbb{M}_{n}(F)$ with Lie
nilpotence index $m$. We must show that $\mathrm{dim}_{F}R\leqslant M(m+1,n)$%
.

Choose field extension $K$ of $F$ such that $K\in{\mathcal{C}}$. By
Proposition 4 ((b) and (c)), the $K$-algebra $R\otimes_{F}K$ is Lie
nilpotent of index $m$ and is isomorphic to a $K$-subalgebra of $\mathbb{M}%
_{n}(K)$. By part (a) of this theorem, $\mathrm{dim}_{K}(R\otimes_{F}K)%
\leqslant M(m+1,n)$. Hence by Proposition 4(a),

\begin{equation*}
\mathrm{dim}_{F}R=\mathrm{dim}_{K}(R\otimes_{F}K)\leqslant M(m+1,n),
\end{equation*}

\noindent as required.
\end{proof}

It follows from Theorem 5 that the Conjecture will hold for a given field $F$%
, if it can be shown to hold for the algebraic closure of $F$. We shall
exploit this fact in the next section.


\begin{proposition}
Every idempotent in a ring satisfying the Engel condition is central.
\end{proposition}

\begin{proof}
If $R$ is an arbitrary ring and $e=e^{2}\in R$, then a routine calculation
shows that for each $a\in R$,

\begin{equation*}
[[1-e,(1-e)a],e]=(1-e)ae.
\end{equation*}

\noindent Putting $\alpha=(1-e)ae$ we see that $\alpha e=\alpha$ and $%
e\alpha=0$ from which it follows that $[\alpha,e]=\alpha$. Iterating, we
obtain

\begin{align*}
&[[\alpha,e],e]=\alpha, \\[0.5ex]
&[[[\alpha,e],e],e]=\alpha,\hspace*{0.6em}\makebox{\rm and in general} \\%
[0.5ex]
&[\alpha,\overset{m}{\overbrace{e,e,\ldots,e}}]^{\ast}=\alpha,\hspace*{0.6em}%
\makebox{\rm for each $m\in\N$.}
\end{align*}

\noindent If $R$ satisfies the Engel condition of index $m$, then we have

\begin{equation*}
\alpha=[\alpha,\overset{m}{\overbrace{e,e,\ldots,e}}]^{\ast}=0,
\end{equation*}

\noindent and so

\begin{align}
(1-e)ae=0.
\end{align}

\noindent Interchanging the roles of $e$ and $1-e$ in the above argument
yields

\begin{align}
ea(1-e)=0.
\end{align}

\noindent Equations (8) and (9) imply $ae-eae=0$ and $ea-eae=0$ whence $%
ea=ae $. We conclude that $e$ is central.
\end{proof}


\begin{proposition}
Every right artinian ring satisfying the Engel condition is isomorphic to a
finite direct product of local rings.
\end{proposition}

\begin{proof}
It is known (see \cite[Theorem 27.6, p.~304]{AF92} or \cite[Theorem 5.9,
p.~49]{Pa91}) that every right artinian ring $R$ contains a complete set of
primitive orthogonal idempotents $\{e_{1},e_{2},\ldots,e_{k}\}$ such that $R$
decomposes as

\begin{equation*}
R_{R}\cong e_{1}R\oplus e_{2}R\oplus \cdots \oplus e_{k}R,
\end{equation*}

\noindent where each $e_{i}R$ has unique maximal proper submodule $e_{i}J(R)$%
. If $R$ satisfies the Engel condition, then each idempotent $e_{i}$ is
central by Proposition 6, so the above decomposition is a decomposition of
(two-sided) ideals with each $e_{i}R=e_{i}Re_{i}$ a local ring.
\end{proof}


\begin{lemma}
Let $F$ be a field and $e$ an idempotent of $\mathbb{M}_{n}(F)$. If $\mathrm{%
rank}\,e=r$, then $e\mathbb{M}_{n}(F)e\cong\mathbb{M}_{r}(F)$ as $F$%
-algebras.
\end{lemma}

\begin{proof}
Since $\mathrm{rank}\,e=r$, $F^{(n)}e$ has dimension $r$ as an $F$-space, so 
$F^{(n)}e\cong F^{(r)}$ as $F$-spaces. Then

\begin{align*}
e\mathbb{M}_{n}(F)e & \cong\mathrm{End}_{F}\left(F^{(n)}e\right) \\[0.5ex]
& \cong\mathrm{End}_{F}\left(F^{(r)}\right) \\[0.5ex]
& \cong \mathbb{M}_{r}(F).
\end{align*}
\end{proof}

The following theorem tells us that for a given field $F$, the Conjecture
will hold for all $F$-subalgebras of $\mathbb{M}_{n}(F)$, if it can be shown
to hold for all local $F$-subalgebras of $\mathbb{M}_{n}(F)$.


\begin{theorem}
The following statements are equivalent for a field $F$:

\begin{list}{{(\alph{standard})}\hfill}{\usecounter{standard}\addtocounter{standard}{0}\setlength{\topsep}{1ex}\setlength{\labelwidth}{1.5\parindent}\setlength{\labelsep}{0.2\parindent}\setlength{\leftmargin}{2.3\parindent}\setlength{\listparindent}{0pt}\setlength{\itemsep}{5pt}}

\item The Conjecture holds in respect of $F$;

\item For all positive integers $m$ and $n$, if $R$ is any {\em local} $F$-subalgebra of $\M_{n}(F)$ with Lie nilpotence index $m$, then

\[
{\rm dim}_{F}R\leqslant M(m+1,n).
\]

\end{list}
\end{theorem}

\begin{proof}
(a)$\Rightarrow$(b) is obvious.

(b)$\Rightarrow$(a) Let $m$ and $n$ be positive integers and $R$ an $F$%
-subalgebra of $\mathbb{M}_{n}(F)$ with Lie nilpotence index $m$. Note that $%
R$ satisfies the Engel condition of index $m$. Since $R$ is a finite
dimensional $F$-algebra, it is right (and left) artinian, and so by
Proposition 7, $R\cong R_{1}\times R_{2}\times\cdots\times R_{k}$ where each 
$R_{i}$ is a local right artinian ring. This entails the existence of a
complete set of central primitive orthogonal idempotents $%
\{e_{1},e_{2},\ldots,e_{k}\}$ in $R$ such that

\begin{align}
R_{R}\cong e_{1}R\oplus e_{2}R\oplus\cdots\oplus e_{k}R
\end{align}

\noindent with $e_{i}R=e_{i}Re_{i}\cong R_{i}$ for each $i\in\{1,2,\ldots,k%
\} $. For each $i\in\{1,2,\ldots,k\}$ put

\begin{align}
r_{i}\,\overset{\mathrm{def}}{=}\mathrm{rank}\,e_{i}.
\end{align}

\noindent The equation $1_{R}=I_{n}=e_{1}+e_{2}+\cdots+e_{k}$ induces the $F$%
-space decomposition

\begin{equation*}
F^{(n)}=F^{(n)}e_{1}\oplus F^{(n)}e_{2}\oplus \cdots \oplus F^{(n)}e_{k}.
\end{equation*}

\noindent Thus

\begin{align}
n & = \mathrm{dim}_{F}F^{(n)}  \notag \\[0.5ex]
& = \mathrm{dim}_{F}(F^{(n)}e_{1})+\mathrm{dim}_{F}(F^{(n)}e_{2})+\cdots+%
\mathrm{dim}_{F}(F^{(n)}e_{k})  \notag \\[0.5ex]
& = r_{1}+r_{2}+\cdots+r_{k}\hspace*{0.6em}\makebox{\rm [by (11)].}
\end{align}

\noindent Observe that each local ring $e_{i}R$ is an $F$-subalgebra of $%
e_{i}\mathbb{M}_{n}(F)e_{i}$, and that $e_{i}\mathbb{M}_{n}(F)e_{i}\cong 
\mathbb{M}_{r_{i}}(F)$ for each $i\in\{1,2,\ldots,k\}$, by Lemma 8. It is
clear too that each $e_{i}R$ must be Lie nilpotent of index $m$, since $R$
has the same property and $e_{i}R\subseteq R$.

The aforementioned facts, together with (b), imply that $\mathrm{dim}%
_{F}(e_{i}R)\leqslant M(m+1,r_{i})$ for each $i\in\{1,2,\ldots, k\}$. Then

\begin{align}
\mathrm{dim}_{F}R & = \sum_{i=1}^{k}\mathrm{dim}_{F}(e_{i}R)\hspace*{0.6em}%
\makebox{\rm [by (10)]}  \notag \\[0.5ex]
& \leqslant\sum_{i=1}^{k}M(m+1,r_{i})  \notag \\[0.5ex]
& \leqslant M\left(m+1,\sum_{i=1}^{k}r_{i}\right)\hspace*{0.6em}%
\makebox{\rm
[by Proposition 28]}  \notag \\[0.5ex]
& = M(m+1,n)\hspace*{0.6em}\makebox{\rm [by (12)].}
\end{align}
\end{proof}


\section{Simultaneous triangularization and the passage to upper triangular
matrix rings}

The main result of this section (Theorem 12) shows that for algebraically
closed fields $F$, the Conjecture reduces to a consideration of $F$%
-subalgebras of $\mathbb{U}_{n}^{\ast}(F)$, the algebra of upper triangular
matrices over $F$ with constant main diagonal.

Recall that an $F$-subalgebra $R$ of $\mathbb{M}_{n}(F)$ is said to be \emph{%
simultaneously upper triangularizable} in $\mathbb{M}_{n}(F)$ if there
exists an invertible $U\in\mathbb{M}_{n}(F)$ such that $U^{-1}RU\subseteq 
\mathbb{U}_{n}(F)$.

A key result is the following. Although implicit in \cite[Theorem 1, p.~434]%
{Ja44} we shall provide a proof in the absence of a suitable reference.


\begin{proposition}
Let $F$ be an algebraically closed field.

\begin{list}{{(\alph{standard})}\hfill}{\usecounter{standard}\addtocounter{standard}{0}\setlength{\topsep}{1ex}\setlength{\labelwidth}{1.1\parindent}\setlength{\labelsep}{0.2\parindent}\setlength{\leftmargin}{2.3\parindent}\setlength{\listparindent}{0pt}\setlength{\itemsep}{5pt}}

\item If $R$ is a finite dimensional local $F$-algebra, then $R$ has $F$-space decomposition $R=F\cdot 1_{R}\oplus J(R)$.

\item If $R$ is a local $F$-subalgebra of $\M_{n}(F)$, then there exists an invertible $U\in\M_{n}(F)$ such that $U^{-1}RU\subseteq \U_{n}^{\ast}(F)$.
Thus, $R$ is isomorphic to an $F$-subalgebra of $\U_{n}^{\ast}(F)$.

\end{list}
\end{proposition}

\begin{proof}
(a) Since $R$ is local, it follows that $R/J(R)$ is a division algebra that
is finite dimensional over $F$. Since $F$ is algebraically closed this
implies $R/J(R)\cong F$. Inasmuch as $F\cdot 1_{R}\cap J(R)=0$, the equation

\begin{align*}
\mathrm{dim}_{F}(F\cdot 1_{R}+J(R)) &=1+\mathrm{dim}_{F}J(R) \\[0.5ex]
&=\mathrm{dim}_{F}(R/J(R))+\mathrm{dim}_{F}J(R) \\[0.5ex]
&=\mathrm{dim}_{F}R
\end{align*}

\noindent entails $R=F\cdot 1_{R}\oplus J(R)$.

(b) It is known (see \cite[Theorem 1.4.6, p.~12]{RR12}) that for an
algebraically closed field $F$, a necessary and sufficient condition for an $%
F$-subalgebra $R$ of $\mathbb{M}_{n}(F)$ to be simultaneously upper
triangularizable in $\mathbb{M}_{n}(F)$ is that $R/J(R)$ is commutative, a
condition that is clearly met in our case. Hence $U^{-1}RU\subseteq \mathbb{U%
}_{n}(F)$ for some invertible $U\in\mathbb{M}_{n}(F)$. Putting $S=U^{-1}RU$
we note that since $S$ is local, $S=FI_{n}\oplus J(S)$ by (a). Since every
element of $J(S)$ is a nilpotent matrix in $\mathbb{U}_{n}(F)$, and a
nilpotent upper triangular matrix is strictly upper triangular, we have

\begin{equation*}
U^{-1}RU=S=FI_{n}\oplus J(S)\subseteq \mathbb{U}_{n}^{\ast}(F).
\end{equation*}
\end{proof}


\begin{remark}
(a) The observation that the factor ring $R/J(R)$ is commutative, is key in
the proof of Proposition 10(b). We point out that this property is possessed
by all Lie nilpotent rings. Indeed, \cite[Proposition 3.1(3), p. 4790]{SvW15}
asserts that if $\mathrm{rad}(R)$ denotes the prime radical of a Lie
nilpotent ring $R$, then $R/\mathrm{rad}(R)$ is commutative. Since $\mathrm{%
rad}(R)\subseteq J(R)$, the commutativity of $R/J(R)$ follows.

(b) If $F$ is an algebraically closed field of characteristic zero (the
latter assumption is not made in Proposition 10) and $R$ is any Lie
nilpotent $F$-subalgebra of $\mathbb{M}_{n}(F)$, then $R$ can be shown to be
simultaneously upper triangularizable in $\mathbb{M}_{n}(F)$ as a
consequence of Lie's Theorem which asserts that if ${\mathfrak{g}}$ is a
finite dimensional solvable Lie algebra with representation $\mathbb{M}%
_{n}(F)$, then ${\mathfrak{g}}$ is simultaneously upper triangularizable in $%
\mathbb{M}_{n}(F)$. Lie's Theorem applies inasmuch as every Lie nilpotent
ring is a nilpotent Lie algebra with respect to the commutator, and
nilpotent Lie algebras are solvable. (This latter fact is explained in the
second open question of Section 9.)
\end{remark}


\begin{theorem}
The following statements are equivalent for an algebraically closed field $F$%
:

\begin{list}{{(\alph{standard})}\hfill}{\usecounter{standard}\addtocounter{standard}{0}\setlength{\topsep}{1ex}\setlength{\labelwidth}{1.5\parindent}\setlength{\labelsep}{0.2\parindent}\setlength{\leftmargin}{2.3\parindent}\setlength{\listparindent}{0pt}\setlength{\itemsep}{5pt}}

\item The Conjecture holds in respect of $F$;

\item For all positive integers $m$ and $n$, if $R$ is any {\em local} $F$-subalgebra of $\M_{n}(F)$ with Lie nilpotence index $m$, then
\[
{\rm dim}_{F}R\leqslant M(m+1,n);
\]

\item For all positive integers $m$ and $n$, if $R$ is any $F$-subalgebra of $\U_{n}^{\ast}(F)$ with Lie nilpotence index $m$, then
\[
{\rm dim}_{F}R\leqslant M(m+1,n).
\]

\end{list}
\end{theorem}

\begin{proof}
(a) and (b) are equivalent by Theorem 9 without any restriction on the field 
$F$.

The equivalence of (b) and (c) is a consequence of Proposition 10(b) which
tells us that up to isomorphism, the local $F$-subalgebras of $\mathbb{M}%
_{n}(F)$ are precisely the $F$-subalgebras of $\mathbb{U}_{n}^{\ast}(F)$.
\end{proof}


\section{Subalgebras of $\mathbb{U}_{n}^{\ast}(F)$}

The main body of theory is developed in this section.

Throughout this section and unless otherwise stated, $F$ shall denote a
field and $R$ an $F$-subalgebra of $\mathbb{U}_{n}^{\ast}(F)$.

Let $V$ be a faithful right $R$-module. We define a sequence $\{R_{k}\}_{k\in%
\mathbb{N}}$ of $F$-subalgebras of $R$, a sequence $\{J_{k}\}_{k\in\mathbb{N}%
}$ where each $J_{k}$ is an ideal of $R_{k}$, and a sequence $\{U_{k}\}_{k\in%
\mathbb{N}}$ of $F$-subspaces of $V$ as follows

\begin{equation*}
\begin{cases}
R_{1}\overset{\mathrm{def}}{=} R, \\[1ex] 
J_{1}\overset{\mathrm{def}}{=} J(R_{1}),\hspace*{0.3em}\makebox{\rm and} \\%
[1ex] 
U_{1}\overset{\mathrm{def}}{=}\hspace*{0.2em}%
\makebox{\rm any $F$-subspace
complement of $VJ_{1}$ in $V$.}%
\end{cases}%
\end{equation*}

\medskip\medskip

\noindent For $k\in\mathbb{N}$, $k\geqslant 2$, define

\begin{align}
\begin{cases}
R_{k}\overset{\mathrm{def}}{=} FI_{n} + (0:^{R_{k-1}}U_{k-1}), \\[1ex] 
J_{k}\overset{\mathrm{def}}{=} J(R_{k}),\hspace*{0.3em}\makebox{\rm and} \\%
[1ex] 
U_{k}\overset{\mathrm{def}}{=}\hspace*{0.2em}%
\parbox[t]{20em}{any $F$-subspace complement of
 $VJ_{1}J_{2}\ldots J_{k}$ in\\[0.2\baselineskip]
 $VJ_{1}J_{2}\ldots J_{k-1}$.}%
\end{cases}%
\end{align}

\medskip

\noindent It follows from the definition of $U_{k}$ that

\begin{align}
VJ_{1}J_{2} \ldots J_{k-1} = U_{k} \oplus VJ_{1}J_{2} \ldots J_{k}
\end{align}

\noindent as $F$-spaces.

For convenience we put $J_{0}=R$.

Since $(0:^{R_{k-1}} U_{k-1})\subseteq R_{k-1}$ and since every $F$%
-subalgebra of $\mathbb{U}_{n}^{\ast}(F)$ contains $FI_{n}$, it is clear
from the definition of $R_{k}$ in (14) that $R_{k-1} \supseteq R_{k}$ for
every $k\in\mathbb{N}$, $k\geqslant 2$. We thus have

\begin{align}
R_{1}\supseteq R_{2}\supseteq\cdots
\end{align}

\noindent It is easily shown that if $S$ and $T$ are any $F$-subalgebras of $%
\mathbb{U}_{n}^{\ast}(F)$, then $S \subseteq T$ if and only if $J(S)
\subseteq J(T)$. In the light of this observation, (16) implies that

\begin{align}
J_{1}\supseteq J_{2}\supseteq\cdots
\end{align}

\noindent Since $J_{k} \subseteq J_{1}$ for all $k\in \mathbb{N}$, and $%
J_{1} $ is nilpotent, we must have $J_{0}J_{1} \ldots J_{k}=0$ for $k$
sufficiently large. Define

\begin{align}
\ell\overset{\mathrm{def}}{=}\min\hspace*{0.1em}\{k\in\mathbb{N}:
J_{0}J_{1}\ldots J_{k}=0\}.
\end{align}

\noindent It follows from (17) that $J_{0}J_{1}\ldots J_{k-1}\supseteq
J_{0}J_{1}\ldots J_{k}$ for each $k\in\mathbb{N}$. We thus have the
descending chain

\begin{equation*}
R=J_{0}\supseteq J_{0}J_{1}\supseteq\cdots\supseteq J_{0}J_{1}\ldots
J_{\ell-1}\supseteq J_{0}J_{1}\ldots J_{\ell}=0.
\end{equation*}

\noindent This, in turn, induces a descending chain

\begin{align}
V = VJ_{0}\supseteq V J_{0}J_{1}\supseteq\cdots\supseteq VJ_{0}J_{1}\ldots
J_{\ell-1}\supseteq 0.
\end{align}

\noindent Note that $VJ_{0}J_{1} \ldots J_{\ell-1} \neq 0$ since $J_{0}J_{1}
\ldots J_{\ell-1} \neq 0$ and $V$ is a faithful right $R$-module.

Recall that if $R$ is an arbitrary ring, then a submodule $N$ of a right $R$%
-module $M$ is said to be \emph{superfluous} if

\begin{equation*}
\forall L\leqslant M, \,N+L=M\Rightarrow L=M.
\end{equation*}


\begin{lemma}
If $I$ is a nilpotent ideal of an arbitrary ring $R$ and $M$ is any right $R$%
-module, then $MI$ is a superfluous submodule of $M$.
\end{lemma}

\begin{proof}
Suppose $MI+L=M$ with $L\leqslant M$. Multiplying by $I$ we obtain $%
MI^{2}+LI=MI$, so $MI^{2}+LI+L=MI+L=M$. Continuing in this way, we obtain $%
MI^{k}+L=M$ for all $k\in\mathbb{N}$. Since $I$ is nilpotent this yields,
for $k$ sufficiently large, the equation $MI^{k}+L=M\cdot 0+L=L=M$.
\end{proof}

Important properties of the chain (19) are established in the next lemma.


\begin{lemma}
Let the sequences $\{R_{k}\}_{k\in\mathbb{N}}$, $\{J_{k}\}_{k\in\mathbb{N}}$
and $\{U_{k}\}_{k\in\mathbb{N}}$ be defined as in (14), and positive integer 
$\ell$ defined as in (18). Let $k \in \{1,2,\ldots,\ell\}$. Then:

\begin{list}{{(\alph{standard})}\hfill}{\usecounter{standard}\addtocounter{standard}{0}\setlength{\topsep}{1ex}\setlength{\labelwidth}{1.5\parindent}\setlength{\labelsep}{0.2\parindent}\setlength{\leftmargin}{2.3\parindent}\setlength{\listparindent}{0pt}\setlength{\itemsep}{5pt}}

\item $VJ_{0}J_{1} \ldots J_{k}$ is a superfluous $R_{k}$-submodule of $VJ_{0}J_{1} \ldots J_{k-1}$;

\item $U_{k}R_{k}=VJ_{0}J_{1} \ldots J_{k-1}=U_{k}\oplus \cdots \oplus U_{\ell}$;

\item $J_{k+1}=(0:^{R_{k}}U_{k})$;

\item $VJ_{0}J_{1} \ldots J_{k-1}$ is a faithful right $R_{k}$-module.
\end{list}
\end{lemma}

\begin{proof}
(a) That $VJ_{0}J_{1}\ldots J_{k}$ is a right $R_{k}$-module is a
consequence of the fact that $J_{0}J_{1}\ldots J_{k}$ is an $F$-subspace of $%
R$ that is closed under right multiplication by elements from $R_{k}$.

Since $R_{k}\subseteq R_{k-1}$, every right $R_{k-1}$-module is canonically
a right $R_{k}$-module. In particular, $VJ_{0}J_{1}\ldots J_{k-1}$ is a
right $R_{k}$-module.

It remains to show that $VJ_{0}J_{1}\ldots J_{k}$ is superfluous in $%
VJ_{0}J_{1}\ldots J_{k-1}$. Put $U=VJ_{0}J_{1}\ldots J_{k-1}$. Since $%
J_{k}\subseteq J_{1}$ and $J_{1}$ is nilpotent, $J_{k}$ must also be
nilpotent. It follows from Lemma 13 that $UJ_{k}$ is a superfluous submodule
of $U$, as required.

(b) Since $U_{k}R_{k}\supseteq U_{k}$, it follows from (15) that

\begin{equation*}
VJ_{0}J_{1}\ldots J_{k-1}=U_{k}R_{k}+VJ_{0}J_{1}\ldots J_{k}
\end{equation*}

\noindent where the right-hand-side of the above equation is a sum of $R_{k}$%
-submodules of $VJ_{0}J_{1}\ldots J_{k-1}$. Since $VJ_{0}J_{1}\ldots J_{k}$
is a superfluous $R_{k}$-submodule of $VJ_{0}J_{1}\ldots J_{k-1}$ by (a), we
must have $U_{k}R_{k}=VJ_{0}J_{1}\ldots J_{k-1}$.

To establish the equation $VJ_{0}J_{1}\ldots J_{k-1}=U_{k}\oplus\cdots\oplus
U_{\ell}$, we note first that the $U_{i}$ constitute an independent family
of $F$-subspaces of $V$. This is clear from the definition of the $U_{i}$ in
(14). This means that the sum $U_{k}\oplus\cdots\oplus U_{\ell}$ is indeed a
direct sum of $F$-subspaces. It remains to establish equality.

Since, by (14), $U_{\ell}$ is an $F$-subspace complement of $%
VJ_{0}J_{1}\ldots J_{\ell}$ in $VJ_{0}J_{1}\ldots J_{\ell-1}$, and since $%
VJ_{0}J_{1}\ldots J_{\ell}=0$ by definition of $\ell$, we must have

\begin{equation*}
VJ_{0}J_{1}\ldots J_{\ell-1}=U_{\ell}.
\end{equation*}

\noindent Repeated application of the formula for $U_{k}$ in (14) shows that 
\begin{equation*}
VJ_{0}J_{1}\ldots J_{\ell-2}=U_{\ell-1}\oplus U_{\ell},
\end{equation*}

\noindent and, more generally, that

\begin{equation*}
VJ_{0}J_{1}\ldots J_{k-1}=U_{k}\oplus\cdots\oplus U_{\ell},
\end{equation*}

\noindent as required.

(c) Since $k-1<\ell$, it follows from (b) and the minimality of $\ell$ that $%
U_{k}R_{k}=VJ_{0}J_{1}\ldots J_{k-1}\neq 0$, whence $U_{k}\neq 0$. This
means that $(0:^{R_{k}}U_{k})$ must be a proper right ideal of $R_{k}$ and
so cannot contain any units of~$R_{k}$. Inasmuch as $R_{k}$ is an $F$%
-subalgebra of $\mathbb{U}_{n}^{\ast}(F)$, $(0:^{R_{k}}U_{k})$ must
therefore comprise strictly upper triangular matrices. Since, by (14), $%
R_{k+1}=FI_{n}+(0:^{R_{k}}U_{k})$, we must have $%
J_{k+1}=J(R_{k+1})=(0:^{R_{k}}U_{k})$.

(d) We use induction on $k$. Take $k=1$. Then $VJ_{0}J_{1}\ldots
J_{k-1}=VJ_{0}=V$, which is a faithful $R_{1}$-module by hypothesis. This
establishes the base case.

To establish the inductive step, take $t\in R_{k}$ with $k\geqslant 2$ and
suppose

\begin{align}
(VJ_{0}J_{1}\ldots J_{k-1})t=0.
\end{align}

\noindent Since $VJ_{0}J_{1}\ldots J_{k-1}\neq 0$, $t$ cannot be a unit of $%
R_{k}$, and since $R_{k}$ is local, we must have $t\in J_{k}$. By (c), $%
J_{k}=(0:^{R_{k-1}}U_{k-1})$, so

\begin{align}
U_{k-1}t=0.
\end{align}

\noindent We thus have

\begin{align*}
(VJ_{0}J_{1}\ldots J_{k-2})t &=(U_{k-1}+VJ_{0}J_{1}\ldots J_{k-1})t\hspace*{%
0.6em}\makebox{\rm [by (15)]} \\[0.5ex]
&=0\hspace*{0.6em}\makebox{\rm [by (20) and (21)].}
\end{align*}

\noindent By the inductive hypothesis, $VJ_{0}J_{1}\ldots J_{k-2}$ is a
faithful right $R_{k-1}$-module. Since $t\in J_{k}\subseteq R_{k}\subseteq
R_{k-1}$, the above equation entails $t=0$. We conclude that $%
VJ_{0}J_{1}\ldots J_{k-1}$ is a faithful $R_{k}$-module.
\end{proof}


\begin{remark}
(a) Taking $k=1$ in Lemma 14(b) yields the $F$-subspace decomposition

\begin{align}
V=U_{1}\oplus U_{2}\oplus\cdots\oplus U_{\ell}.
\end{align}

\noindent Substituting the equation $VJ_{0}J_{1}\ldots
J_{k-1}=U_{k}\oplus\cdots\oplus U_{\ell}$ of Lemma 14(b) into (22) yields

\begin{align}
V=U_{1}\oplus\cdots\oplus U_{k-1}\oplus VJ_{0}J_{1}\ldots J_{k-1}.
\end{align}

(b) The faithfulness of $VJ_{0}J_{1}\ldots J_{k-1}$ proved in Lemma 14(d)
means that $(VJ_{0}J_{1}\ldots J_{k-1})J_{k}=0$ if and only if $J_{k}=0$.
Moreover, since $V$ is faithful as a right $R$-module, we have that

\begin{equation*}
VJ_{0}J_{1}\ldots J_{k}=0 \,\Leftrightarrow\, J_{0}J_{1}\ldots J_{k}=0.
\end{equation*}

\noindent It follows that 
\begin{equation*}
J_{0}J_{1}\ldots J_{k}=0 \,\Leftrightarrow\, J_{k}=0.
\end{equation*}

\noindent This has the consequence that

\begin{equation*}
\ell=\min\hspace*{0.1em}\{k\in\mathbb{N}: J_{0}J_{1}\ldots J_{k}=0\}=\min%
\hspace*{0.1em}\{k\in\mathbb{N}: J_{k}=0\}.
\end{equation*}
\end{remark}


\begin{proposition}
Let the sequences $\{R_{k}\}_{k\in\mathbb{N}}$, $\{J_{k}\}_{k\in\mathbb{N}}$
and $\{U_{k}\}_{k\in\mathbb{N}}$ be defined as in (14), and positive integer 
$\ell$ defined as in (18). Then

\begin{equation*}
R=R_{1}\supset R_{2}\supset\cdots\supset R_{\ell}=R_{\ell+1}=\cdots
\end{equation*}

\noindent is a strictly descending chain of $F$-subalgebras of $\mathbb{U}%
_{n}^{\ast}(F)$ that stabilizes at $R_{\ell}$. Moreover, $J_{\ell}=0$, so
that $R_{\ell}=FI_{n}$.
\end{proposition}

\begin{proof}
Suppose $R_{k}=R_{k+1}$ for some $k\leqslant\ell$. Note that we cannot have $%
U_{k}=0$ since this would imply, by Lemma 14(b), that $VJ_{0}J_{1}\ldots
J_{k-1}=0$, which contradicts the fact that $VJ_{0}J_{1}\ldots
J_{k-1}\supseteq VJ_{0}J_{1}\ldots J_{\ell-1}\neq 0$. Now

\begin{align*}
0 &=U_{k}J_{k+1}\hspace*{0.6em}%
\makebox{\rm [because
$J_{k+1}=(0:^{R_{k}}U_{k})$ by Lemma 14(c)]} \\[1ex]
&=U_{k}J_{k}\hspace*{0.6em}%
\makebox{\rm [because
$J_{k}=J(R_{k})=J(R_{k+1})=J_{k+1}$ by hypothesis]} \\[1ex]
&=(U_{k}R_{k})J_{k}\hspace*{0.6em}%
\makebox{\rm [because $J_{k}$ is an ideal
of $R_{k}$]} \\[1ex]
&=(VJ_{0}J_{1}\ldots J_{k-1})J_{k}\hspace*{0.6em}%
\makebox{\rm [by Lemma
14(b)].}
\end{align*}

\noindent Since $k\leqslant\ell$ it follows from the minimality of $\ell$
that $k=\ell$. We have thus proven that $R_{k}\supset R_{k+1}$ for $%
k\in\{1,2,\ldots,\ell-1\}$.

In Remark 15(b) we noted that $J_{\ell}=0$. Since $R_{\ell}\subseteq\mathbb{U%
}_{n}^{\ast}(F)$, this entails $R_{\ell}=FI_{n}$. However, since every $F$%
-subalgebra of $\mathbb{U}_{n}^{\ast}(F)$ contains $FI_{n}$, the descending
chain of $F$-subalgebras must stabilize at $R_{\ell}$.
\end{proof}

Let the sequences $\{R_{k}\}_{k\in\mathbb{N}}$, $\{J_{k}\}_{k\in\mathbb{N}}$
and $\{U_{k}\}_{k\in\mathbb{N}}$ be defined as in (14), and positive integer 
$\ell$ defined as in (18). For each $k\in\{1,2,\ldots,\ell\}$ define

\begin{align}
d_{k}\overset{\mathrm{def}}{=}\mathrm{dim}_{F}U_{k}.
\end{align}

A key step in the proof of Theorem 17(c) below is inspired by \cite[2. Proof
of Schur's Inequality, p. 558]{Gu76}.


\begin{theorem}
Let the sequences $\{R_{k}\}_{k\in\mathbb{N}}$, $\{J_{k}\}_{k\in\mathbb{N}}$
and $\{U_{k}\}_{k\in\mathbb{N}}$ be defined as in (14), positive integer $%
\ell$ defined as in (18), and $\{d_{k}:1\leqslant k\leqslant \ell\}$ defined
as in (24). Then:

\begin{list}{{(\alph{standard})}\hfill}{\usecounter{standard}\addtocounter{standard}{0}\setlength{\topsep}{1ex}\setlength{\labelwidth}{1.5\parindent}\setlength{\labelsep}{0.2\parindent}\setlength{\leftmargin}{2.3\parindent}\setlength{\listparindent}{0pt}\setlength{\itemsep}{5pt}}

\item ${\displaystyle{\rm dim}_{F}(U_{k}J_{k})={\rm dim}_{F}V-\sum_{i=1}^{k}d_{i}}$ for each $k\in\{1,2,\ldots,\ell\}$.

\item ${\displaystyle{\rm dim}_{F}V=\sum_{i=1}^{\ell}d_{i}}$.\\[0.5ex]


\item ${\rm dim}_{F}R \leqslant M(\ell,{\rm dim}_{F}V)$.
\end{list}
\end{theorem}

\begin{proof}
(a) Inasmuch as

\begin{align*}
V &=U_{1}\oplus\cdots\oplus U_{k}\oplus VJ_{0}J_{1}\ldots J_{k}\hspace*{0.6em%
}\makebox{\rm [by (23)]} \\[0.5ex]
&=U_{1}\oplus\cdots\oplus U_{k}\oplus (VJ_{0}J_{1}\ldots J_{k-1})J_{k} \\%
[0.5ex]
&=U_{1}\oplus\cdots\oplus U_{k}\oplus (U_{k}R_{k})J_{k}\hspace*{0.6em}%
\makebox{\rm [by Lemma 14(b)]} \\[0.5ex]
&=U_{1}\oplus\cdots\oplus U_{k}\oplus U_{k}J_{k}\hspace*{0.6em}%
\makebox{\rm
[because $J_{k}$ is an ideal of $R_{k}$],}
\end{align*}

\noindent we have $\mathrm{dim}_{F}V=d_{1}+\cdots+d_{k}+\mathrm{dim}%
_{F}(U_{k}J_{k})$, from which (a) follows.

(b) is an immediate consequence of (22) and (24).

(c) If $\ell=1$, then $J=J_{\ell}=0$, so $\mathrm{dim}_{F}R=1=M(1,\mathrm{dim%
}_{F}V)$ and there is nothing further to prove. Suppose $\ell\geqslant 2$.

We next derive the recursive formula

\begin{align}
\mathrm{dim}_{F}J_{k}\leqslant d_{k}\left(\mathrm{dim}_{F}V-%
\sum_{i=1}^{k}d_{i}\right)+\mathrm{dim}_{F}J_{k+1}\hspace*{0.6em}(1\leqslant
k\leqslant\ell).
\end{align}

\noindent To this end, take $k\in\{1,2,\ldots,\ell\}$, $X\in J_{k}$ and let $%
\rho_{X}:U_{k}\rightarrow U_{k}J_{k}$ be the right multiplication by $X$
map. Observe that $\rho_{X}$ is an $F$-linear map and thus a member of $%
\mathrm{Hom}_{F}(U_{k},U_{k}J_{k})$.

Define the map $\Theta: J_{k} \rightarrow \mathrm{Hom}_{F}(U_{k},U_{k}J_{k})$
by $\Theta(X)=\rho_{X}$. It is also easily seen that $\Theta$ is an $F$%
-linear map. Note that

\begin{align}
\mathrm{Ker}\,\Theta & = \{X\in J_{k}:\rho_{X}=0\}  \notag \\[0.5ex]
& = \{X\in J_{k}:U_{k}X=0\}  \notag \\[0.5ex]
& = (0:^{J_{k}}U_{k})  \notag \\[0.5ex]
& = J_{k}\cap (0:^{R_{k}}U_{k})  \notag \\[0.5ex]
& =J_{k}\cap J_{k+1}\hspace*{0.6em}\makebox{\rm [by Lemma 14(c)]}  \notag \\%
[0.5ex]
& =J_{k+1}\hspace*{0.6em}\makebox{\rm [because $J_{k}\supseteq J_{k+1}$].}
\end{align}

\noindent We thus have

\begin{align*}
\mathrm{dim}_{F}J_{k} &=\mathrm{rank}\,\Theta+\mathrm{nullity}\,\Theta \\%
[0.5ex]
&\leqslant \mathrm{dim}_{F}\left(\mathrm{Hom}_{F}(U_{k},U_{k}J_{k})\right)+%
\mathrm{dim}_{F}J_{k+1}\hspace*{0.6em}\makebox{\rm [by (26)]} \\[0.5ex]
&=\mathrm{dim}_{F}U_{k}\cdot\mathrm{dim}_{F}(U_{k}J_{k})+\mathrm{dim}%
_{F}J_{k+1} \\[0.5ex]
&=d_{k}\left(\mathrm{dim}_{F}V-{\displaystyle\sum_{i=1}^{k}d_{i}}\right)+%
\mathrm{dim}_{F}J_{k+1}\hspace*{0.6em}\makebox{\rm [by (a)],}
\end{align*}

\noindent which is (25).

Letting $k$ take on the values from $1$ to $\ell-1$ in (25), we see that

\begin{align}
\mathrm{dim}_{F}J_{1} &\leqslant\sum_{j=1}^{\ell-1}d_{j}\left(\mathrm{dim}%
_{F}V-\sum_{i=1}^{j}d_{i}\right)+\mathrm{dim}_{F}J_{\ell}  \notag \\[0.5ex]
&=\sum_{j=1}^{\ell-1}d_{j}\left(\mathrm{dim}_{F}V-\sum_{i=1}^{j}d_{i}\right)%
\hspace*{0.6em}\makebox{\rm [because $J_{\ell}=0$]}  \notag \\[0.5ex]
&=\frac{1}{2}\left((\mathrm{dim}_{F}V)^{2}-\sum_{i=1}^{\ell}d_{i}^{2}\right)%
\hspace*{0.6em}%
\makebox{\rm [because ${\rm dim}_{F}V=\sum_{i=1}^{\ell}d_{i}$
by (b)]}  \notag \\[0.5ex]
&\leqslant 
\parbox[t]{24em}{$M(\ell,{\rm dim}_{F}V)-1$\hspace*{0.6em}[by the definition of $M(\ell,{\rm dim}_{F}V)$ \\[0.5ex]
\,noting that ${\rm dim}_{F}V=\sum_{i=1}^{\ell}d_{i}$].}
\end{align}

\noindent Since $R$ has $F$-space decomposition $R=FI_{n}\oplus J$, we have

\begin{align*}
\mathrm{dim}_{F}R &=1+\mathrm{dim}_{F}J \\[0.5ex]
&=1+\mathrm{dim}_{F}J_{1}\hspace*{0.6em}\makebox{\rm [because $J=J_{1}$]} \\%
[0.5ex]
&\leqslant 1+M(\ell,\mathrm{dim}_{F}V)-1\hspace*{0.6em}%
\makebox{\rm [by
(27)]} \\[0.5ex]
&=M(\ell,\mathrm{dim}_{F}V).
\end{align*}
\end{proof}

In Proposition 29 it is shown that $M(\ell,n)$ is an increasing function in
both arguments. This means, with reference to Theorem 17(c), that the
smaller the value of $\ell$, the lower the upper bound $M(\ell,\mathrm{dim}%
_{F}V)$ for $\mathrm{dim}_{F}R$.

We shall show presently that if the $F$-subalgebra $R$ of $\mathbb{U}%
_{n}^{\ast}(F)$ has radical $J$ satisfying $J^{m}=0$ for some $m\in\mathbb{N}
$, then the value of $\ell$ cannot exceed $m$, and so

\begin{equation*}
\mathrm{dim}_{F}R\leqslant M(m,\mathrm{dim}_{F}V).
\end{equation*}

\noindent In the next section we shall strengthen the above by proving that
if $R$ has Lie nilpotence index $m$ (this is the case if $J^{m+1}=0$), then
the value of $\ell$ cannot exceed $m+1$, from which we may deduce

\begin{equation*}
\mathrm{dim}_{F}R\leqslant M(m+1,\mathrm{dim}_{F}V).
\end{equation*}

Since the $d_{i}$ are positive in Theorem 17(b), we must have $\ell\leqslant%
\mathrm{dim}_{F}V$. A combination of Theorem 17(c), the fact that $M(\ell,n)$
is increasing in its first argument (Proposition 29), and the formula for $%
M(n,n)$ derived in Corollary 27(a), yields:


\begin{corollary}
If $R$ is an $F$-subalgebra of $\mathbb{U}_{n}^{\ast}(F)$ and $V$ any
faithful right $R$-module, then

\begin{equation*}
\mathrm{dim}_{F}R\leqslant M(\mathrm{dim}_{F}V,\mathrm{dim}_{F}V)={%
\textstyle \frac{1}{2}}\left((\mathrm{dim}_{F}V)^{2}-\mathrm{dim}%
_{F}V\right)+1.
\end{equation*}
\end{corollary}


\begin{remark}
If $V=F^{n}=\overbrace{F\times F\times\cdots\times F}^{n\hspace*{0.2em}%
\mathrm{times}}$ is interpreted as a $1\times n$ matrix over $F$, then it
has the canonical structure of a faithful right module with respect to any $%
F $-subalgebra of the matrix algebra $\mathbb{M}_{n}(F)$. For such a module $%
V$, we have

\begin{equation*}
\mathrm{dim}_{F}V=n.
\end{equation*}

\noindent This allows us to replace $\mathrm{dim}_{F}V$ with $n$ in each of
the results in this, and subsequent, sections. In particular, taking $%
\mathrm{dim}_{F}V=n$ in the previous corollary yields the upper bound

\begin{equation*}
\mathrm{dim}_{F}R\leqslant{\textstyle \frac{1}{2}}(n^{2}-n)+1,
\end{equation*}

\noindent an observation that has little value, since the expression $\frac{1%
}{2}(n^{2}-n)+1$ coincides with the dimension of the overlying $F$-algebra $%
\mathbb{U}_{n}^{\ast}(F)$.
\end{remark}


\begin{proposition}
Let the sequences $\{R_{k}\}_{k\in\mathbb{N}}$, $\{J_{k}\}_{k\in\mathbb{N}}$
and $\{U_{k}\}_{k\in\mathbb{N}}$ be defined as in (14), and positive integer 
$\ell$ defined as in (18). If $J^{m}=0$ for some $m\in\mathbb{N}$, then $%
\ell\leqslant m$.
\end{proposition}

\begin{proof}
Inasmuch as $J_{0}J_{1}\ldots J_{m}\subseteq J^{m}=0$, it follows from the
definition of $\ell$ in (18) that $\ell\leqslant m$.
\end{proof}


\begin{corollary}
If $R$ is an $F$-subalgebra of $\mathbb{U}_{n}^{\ast}(F)$ satisfying $%
J^{m}=0 $, and $V$ is any faithful right $R$-module, then

\begin{equation*}
\mathrm{dim}_{F}R\leqslant M(m,\mathrm{dim}_{F}V).
\end{equation*}
\end{corollary}

\begin{proof}
It follows from Theorem 17(c) and Proposition 20 that there exists a
positive integer $\ell\leqslant m$ such that $\mathrm{dim}_{F}R\leqslant
M(\ell,\mathrm{dim}_{F}V)$. By Proposition 29, $M(\ell,\mathrm{dim}%
_{F}V)\leqslant M(m,\mathrm{dim}_{F}V)$, whence $\mathrm{dim}_{F}R\leqslant
M(m,\mathrm{dim}_{F}V)$.
\end{proof}


\section{Lie nilpotent subalgebras of $\mathbb{U}_{n}^{\ast}(F)$: the main
theorem}

A routine inductive argument establishes the following.


\begin{lemma}
Let $R$ be an arbitrary ring and $\{r_{i}:1\leqslant i\leqslant m\}\subseteq
R$. Then

\begin{equation*}
[r_{1},r_{2},\ldots,r_{m}]^{\ast}=\sum_{\sigma\in
S_{m}}c_{\sigma}r_{\sigma(1)}r_{\sigma(2)}\ldots r_{\sigma(m)}
\end{equation*}

\noindent where $c_{\sigma}\in\{-1,0,1\}$ for all $\sigma\in S_{m}$, and $%
\{\sigma\in S_{m}: c_{\sigma}\neq 0$ and $\sigma(1)=1\}$ is a singleton
comprising the identity permutation.
\end{lemma}


\begin{proposition}
Let the sequences $\{R_{k}\}_{k\in\mathbb{N}}$, $\{J_{k}\}_{k\in\mathbb{N}}$
and $\{U_{k}\}_{k\in\mathbb{N}}$ be defined as in (14), and positive integer 
$\ell$ defined as in (18). If $R$ is Lie nilpotent of index $m$, then $%
\ell\leqslant m+1$.
\end{proposition}

\begin{proof}
By Lemma 14(c), we have $J_{2}=(0:^{R}U_{1})$. Pick arbitrary $r\in R$ and $%
b_{k}\in J_{k}$ for each $k\in\{2,\ldots,m+1\}$. Since $U_{1}J_{2}=0$ and $%
J_{2}\supseteq J_{3}\supseteq\cdots\supseteq J_{m+1}$, we have $U_{1}b_{k}=0$
for all $k\in\{2,\ldots,m+1\}$. Thus, using Lemma 22, we see that $%
U_{1}[r,b_{2},\ldots,b_{m+1}]^{\ast}=U_{1}rb_{2}\ldots b_{m+1}$. But $R$ is
Lie nilpotent of index $m$, so $[r,b_{2},\ldots,b_{m+1}]^{\ast}=0$ whence $%
U_{1}rb_{2}\ldots b_{m+1}=0$. Since $r$ is arbitrary, we get

\begin{align*}
0 &=(U_{1}R)b_{2}\ldots b_{m+1} \\[0.5ex]
&=V b_{2}\ldots b_{m+1}\hspace*{0.6em}%
\makebox{\rm [because $U_{1}R=V$ by
Lemma 14(b)],} \\[0.5ex]
\end{align*}

\noindent from which we infer $b_{2}\ldots b_{m+1}=0$ since $V$ is faithful.
It follows that $J_{2}\ldots J_{m+1}=0$, so $\ell\leqslant m+1$ by
definition of $\ell$.
\end{proof}


\begin{theorem}
For all positive integers $m$ and $n$, and fields $F$, if $R$ is any $F$%
-subalgebra of $\mathbb{U}_{n}^{\ast}(F)$ with Lie nilpotence index $m$, then

\begin{equation*}
\mathrm{dim}_{F}R\leqslant M(m+1,n).
\end{equation*}
\end{theorem}

\begin{proof}
Let $m$ and $n$ be arbitrary positive integers, and $F$ an arbitrary field.
Let $R$ be an $F$-subalgebra of $\mathbb{U}_{n}^{\ast}(F)$ with Lie
nilpotence index $m$. If sequences $\{R_{k}\}_{k\in\mathbb{N}}$, $%
\{J_{k}\}_{k\in\mathbb{N}}$ and $\{U_{k}\}_{k\in\mathbb{N}}$ are defined as
in (14), and positive integer $\ell$ defined as in (18), then it follows
from Theorem 17(c) that

\begin{equation*}
\mathrm{dim}_{F}R\leqslant M(\ell,\mathrm{dim}_{F}V).
\end{equation*}

\noindent Choose $V$ to be $F^{n}$, so that $\mathrm{dim}_{F}V=n$ (see
Remark 19). By Proposition 23, $\ell\leqslant m+1$. Since $M(\ell,n)$ is
increasing in its first argument by Proposition 29, we have

\begin{equation*}
\mathrm{dim}_{F}R\leqslant M(\ell,\mathrm{dim}_{F}V)=M(\ell,n)\leqslant
M(m+1,n).
\end{equation*}
\end{proof}


\begin{remark}
Let $R$ be any $F$-subalgebra of $\mathbb{U}_{n}^{\ast}(F)$ satisfying the
polynomial identity

\begin{equation*}
f(x_{1},x_{2},\ldots,x_{m})=\sum_{\sigma\in
S_{m}}c_{\sigma}x_{\sigma(1)}x_{\sigma(2)}\ldots x_{\sigma(m)}=0
\end{equation*}

\noindent where $c_{\sigma}\in F$ for all $\sigma\in S_{m}$, and $%
\{\sigma\in S_{m}: c_{\sigma}\neq 0$ and $\sigma(1)=1\}$ is a singleton
comprising the identity permutation.

Arguments similar to those used earlier in this section show that

\begin{equation*}
\mathrm{dim}_{F}R\leqslant M(m+1,n).
\end{equation*}
\end{remark}

We are finally in a position to complete the proof of the Conjecture.\medskip

\begin{quote}
\textbf{Proof of Conjecture.} \,Let $F$ be any field with algebraic closure $%
K$. Taking the field $F$ of Theorems 12 and 24 to be $K$, we see that the
latter is just Statement (c) of the former. It thus follows from Theorem 12
((c)$\Rightarrow$(a)) that the Conjecture holds in respect of field $K$.

Taking the class of fields ${\mathcal{C}}$ in Theorem 5 to be the singleton $%
{\mathcal{C}}=\{K\}$ and noting that $F$ is a subfield of $K$, we conclude
that the Conjecture holds in respect of field $F$. Since $F$ was chosen
arbitrarily, the proof is complete.
\end{quote}


\section{The function $M(\ell,n)$}

The purposes of this section are twofold. First, to establish a number of
important properties of the function $M(\ell,n)$ that are required in
earlier theory, and second to obtain an explicit description of $M(\ell,n)$;
without such a description, the important results of this paper remain
somewhat opaque. This task will involve the solution of an integer-variable
optimization problem. Our methods, however, are first principled and require
no background knowledge of integer optimization techniques.

We shall make use of the following notation: if ${\mathbf{k}}%
=(k_{1},k_{2},\ldots,k_{\ell})\in\mathbb{N}_{0}^{\ell}$, then:

\begin{list}{{(\alph{standard})}\hfill}{\usecounter{standard}\addtocounter{standard}{0}\setlength{\topsep}{1ex}\setlength{\labelwidth}{1.5\parindent}\setlength{\labelsep}{0.2\parindent}\setlength{\leftmargin}{2.3\parindent}\setlength{\listparindent}{0pt}\setlength{\itemsep}{5pt}}

\item[$\triangleright$] \hspace*{0.5em}${\rm supp}\,{\mathbf k}\,\define\{i\in\{1,2,\ldots,\ell\}:k_{i}>0\}$; and

\item[$\triangleright$] \hspace*{0.5em}${\displaystyle
|{\mathbf k}|\,\define\left(\sum_{i=1}^{\ell}k_{i}^{2}\right)^{1/2}}$ so that ${\displaystyle |{\mathbf k}|^{2}=\sum_{i=1}^{\ell}k_{i}^{2}}$.
\end{list}


\begin{proposition}
Let $\ell$ and $n$ be positive integers. Then the following statements are
equivalent for ${\mathbf{k}}=(k_{1},k_{2},\ldots,k_{\ell})\in\mathbb{N}%
_{0}^{\ell}$ such that $\sum_{i=1}^{\ell}k_{i}=n$:

\begin{list}{{(\alph{standard})}\hfill}{\usecounter{standard}\addtocounter{standard}{0}\setlength{\topsep}{1ex}\setlength{\labelwidth}{1.5\parindent}\setlength{\labelsep}{0.2\parindent}\setlength{\leftmargin}{2.3\parindent}\setlength{\listparindent}{0pt}\setlength{\itemsep}{5pt}}

\item $M(\ell,n)=\frac{1}{2}\left(n^{2}-|{\mathbf k}|^{2}\right)+1$;

\item $|k_{i}-k_{j}|\leqslant 1$ for all $i,j\in\{1,2,\ldots,\ell\}$.
\end{list}
\end{proposition}

\begin{proof}
(a)$\Rightarrow$(b) Suppose (a) holds but $|k_{p}-k_{q}|\geqslant 2$ for
some $p,q\in\{1,2,\ldots,\ell\}$. Without loss of generality, we may suppose
that $k_{p}\geqslant k_{q}+2$. Define ${\mathbf{k}}^{\prime}=(k_{1}^{%
\prime},k_{2}^{\prime},\ldots,k_{\ell}^{\prime})\in\mathbb{N}_{0}^{\ell}$ by:

\begin{equation*}
k_{i}^{\prime}\overset{\mathrm{def}}{=} 
\begin{cases}
k_{i},\hspace*{0.4em}\makebox{\rm if $i\notin\{p,q\}$} \\[1ex] 
k_{p}-1,\hspace*{0.4em}\makebox{\rm if $i=p$} \\[1ex] 
k_{q}+1,\hspace*{0.4em}\makebox{\rm if $i=q$.}%
\end{cases}%
\end{equation*}

\noindent Note that $\sum_{i=1}^{\ell}k_{i}^{\prime}=\sum_{i=1}^{%
\ell}k_{i}=n $. Then:

\begin{align*}
&{\textstyle\frac{1}{2}}\left(n^{2}-|{\mathbf{k}}^{\prime}|^{2}\right)+1-M(%
\ell,n) \\[1ex]
&= {\textstyle \frac{1}{2}}\left(n^{2}-|{\mathbf{k}}^{\prime}|^{2}\right)+1-%
\left({\textstyle \frac{1}{2}}\left(n^{2}-|{\mathbf{k}}|^{2}\right)+1\right)
\\
&= \frac{1}{2}\left(\sum_{i=1}^{\ell}\left(k_{i}^{2}-(k_{i}^{\prime})^{2}%
\right)\right) \\[0.5ex]
&= {\textstyle \frac{1}{2}}\left(k_{p}^{2}+k_{q}^{2}-(k_{p}^{%
\prime})^{2}-(k_{q}^{\prime})^{2}\right) \\[1ex]
&= {\textstyle \frac{1}{2}}%
\left(k_{p}^{2}+k_{q}^{2}-(k_{p}-1)^{2}-(k_{q}+1)^{2}\right) \\[1ex]
&= {\textstyle \frac{1}{2}}\left(2k_{p}-2k_{q}-2\right) \\[1ex]
&= k_{p}-k_{q}-1>0\hspace*{0.6em}%
\makebox{\rm [because $k_{p}\geqslant
k_{q}+2$].}
\end{align*}

\noindent This implies that $\frac{1}{2}\left(n^{2}-|{\mathbf{k}}%
^{\prime}|^{2}\right)+1>M(\ell,n)$, a contradiction.

(b)$\Rightarrow$(a) Suppose ${\mathbf{k}}=(k_{1},k_{2},\ldots,k_{\ell})\in%
\mathbb{N}_{0}^{\ell}$ is such that $\sum_{i=1}^{\ell}k_{i}=n$ and $%
|k_{i}-k_{j}|\leqslant 1$ for all $i,j\in\{1,2,\ldots,\ell\}$. Inasmuch as
each $k_{i}$ is nonnegative this implies the existence of some $r\in\mathbb{N%
}$ such that

\begin{align}
k_{i}\in\{r-1,r\}\hspace*{0.5em}\forall i\in\{1,2,\ldots,\ell\}.
\end{align}

\noindent Now suppose $M(\ell,n)=\frac{1}{2}\left(n^{2}-|{\mathbf{k}}%
^{\prime}|^{2}\right)+1$ with ${\mathbf{k}}^{\prime}=(k_{1}^{\prime},k_{2}^{%
\prime},\ldots,k_{\ell}^{\prime})\in\mathbb{N}_{0}^{\ell}$ such that $%
\sum_{i=1}^{\ell}k_{i}^{\prime}=n$. It follows from implication (a)$%
\Rightarrow$(b) that $|k_{i}^{\prime}-k_{j}^{\prime}|\leqslant 1$ for all $%
i,j\in\{1,2,\ldots,\ell\}$, so there must exist some $s\in\mathbb{N}$ such
that

\begin{align}
k_{i}^{\prime}\in\{s-1,s\}\hspace*{0.5em}\forall i\in\{1,2,\ldots,\ell\}.
\end{align}

\noindent If $r<s$, then it follows from (28) and (29) that

\begin{equation*}
k_{i}\leqslant r\leqslant s-1\leqslant k_{i}^{\prime}\hspace*{0.5em}\forall
i\in\{1,2,\ldots,\ell\}.
\end{equation*}

\noindent Since $\sum_{i=1}^{\ell}k_{i}=\sum_{i=1}^{\ell}k_{i}^{\prime}$,
the above inequalities can only be satisfied if $k_{i}=k_{i}^{\prime}$ for
all $i\in\{1,2,\ldots,\ell\}$, whence ${\mathbf{k}}={\mathbf{k}}^{\prime}$.

A similar argument shows that ${\mathbf{k}}={\mathbf{k}}^{\prime}$ whenever $%
r>s$. Thus if $r\neq s$, then ${\mathbf{k}}={\mathbf{k}}^{\prime}$, whence $%
\frac{1}{2}\left(n^{2}-|{\mathbf{k}}|^{2}\right)+1=\frac{1}{2}\left(n^{2}-|{%
\mathbf{k}}^{\prime}|^{2}\right)+1=M(\ell,n)$ and the proof is complete.

Now suppose $r=s$. Since $k_{i},k_{i}^{\prime}\in\{r,r-1\}$ for each $%
i\in\{1,2,\ldots,\ell\}$ and since $\sum_{i=1}^{\ell}k_{i}=\sum_{i=1}^{%
\ell}k_{i}^{\prime}$, it is easily seen that ${\mathbf{k}}$ and ${\mathbf{k}}%
^{\prime}$ are equal to within permutation of their coordinates, that is to
say, there exists a permutation $\sigma\in S_{\ell}$ such that $%
k_{i}^{\prime}=k_{\sigma(i)}$ for all $i\in\{1,2,\ldots,\ell\}$. Clearly, in
such a situation $|{\mathbf{k}}|=|{\mathbf{k}}^{\prime}|$ and $\frac{1}{2}%
\left(n^{2}-|{\mathbf{k}}|^{2}\right)+1=\frac{1}{2}\left(n^{2}-|{\mathbf{k}}%
^{\prime}|^{2}\right)+1=M(\ell,n)$.
\end{proof}


\begin{corollary}
Let $\ell$ and $n$ be positive integers. Then:

\begin{list}{{(\alph{standard})}\hfill}{\usecounter{standard}\addtocounter{standard}{0}\setlength{\topsep}{1ex}\setlength{\labelwidth}{1.5\parindent}\setlength{\labelsep}{0.2\parindent}\setlength{\leftmargin}{2.3\parindent}\setlength{\listparindent}{0pt}\setlength{\itemsep}{5pt}}

\item If $\ell\leqslant n$, then $M(\ell,n)=\frac{1}{2}\left(n^{2}-|{\mathbf k}|^{2}\right)+1$ for some
${\mathbf k}=(k_{1},k_{2},\ldots,k_{\ell})\in\N_{0}^{\ell}$ with $k_{i}\geqslant 1$ for all $i\in\{1,2,\ldots,\ell\}$.
In particular, $M(n,n)=\frac{1}{2}\left(n^{2}-n\right)+1$.

\item If $\ell>n$, then $M(\ell,n)=M(n,n)=\frac{1}{2}\left(n^{2}-n\right)+1$.
\end{list}
\end{corollary}

\begin{proof}
By Proposition 26, we can choose ${\mathbf{k}}=(k_{1},k_{2},\ldots,k_{\ell})%
\in\mathbb{N}_{0}^{\ell}$ such that $\sum_{i=1}^{\ell}k_{i}=n$, $M(\ell,n)=%
\frac{1}{2}\left(n^{2}-|{\mathbf{k}}|^{2}\right)+1$ and $|k_{i}-k_{j}|%
\leqslant 1$ for all $i,j\in\{1,2,\ldots,\ell\}$.

(a) Suppose $\ell\leqslant n$. If $k_{j}=0$ for some $j\in\{1,2,\ldots,\ell%
\} $, then $k_{i}\in\{0,1\}$ for all $i\in\{1,2,\ldots,\ell\}$, whence $%
n=\sum_{i=1}^{\ell}k_{i}<\ell\leqslant n$, a contradiction.

If $\ell=n$, then clearly $k_{i}=1$ for all $i\in\{1,2,\ldots,\ell\}$, so $|{%
\mathbf{k}}|^{2}=n$ and $M(\ell,n)=M(n,n)=\frac{1}{2}\left(n^{2}-n\right)+1$.

(b) Suppose $\ell> n$. Since $n=\sum_{i=1}^{\ell}k_{i}$, we must have $%
k_{j}=0$ for some $j\in\{1,2,\ldots,\ell\}$. Thus $k_{i}\in\{0,1\}$ for all $%
i\in\{1,2,\ldots,\ell\}$, so $|{\mathbf{k}}|^{2}=n$ and $M(\ell,n)=M(n,n)$.
\end{proof}


\begin{proposition}
Let $\ell$ be an integer satisfying $\ell\geqslant 2$ and $%
n_{1},n_{2},\ldots,n_{k}$ any sequence of positive integers. Then 
\begin{equation*}
M\!\left(\ell,\sum_{i=1}^{k}n_{i}\right)\geqslant
\sum_{i=1}^{k}M(\ell,n_{i}).
\end{equation*}
\end{proposition}

\begin{proof}
We provide a proof in the case $k=2$; the arguments used can be applied
mutatis-mutandis to establish the inductive step in a proof by induction on $%
k$. Choose ${\mathbf{k}}=(k_{1},k_{2},\ldots,k_{\ell})\in\mathbb{N}%
_{0}^{\ell}$ such that

\begin{align}
\sum_{i=1}^{\ell}k_{i}=n_{1}
\end{align}

\noindent and

\begin{align}
M(\ell,n_{1})={\textstyle \frac{1}{2}}\left(n_{1}^{2}-|{\mathbf{k}}%
|^{2}\right)+1,
\end{align}

\noindent and choose $\overline{\mathbf{k}}=(\bar{k}_{1},\bar{k}_{2},\ldots,%
\bar{k}_{\ell})\in\mathbb{N}_{0}^{\ell}$ such that

\begin{align}
\sum_{i=1}^{\ell}\bar{k}_{i}=n_{2}
\end{align}

\noindent and

\begin{align}
M(\ell,n_{2})={\textstyle \frac{1}{2}}\left(n_{2}^{2}-|\overline{\mathbf{k}}%
|^{2}\right)+1.
\end{align}

If $|\mathrm{supp}\,{\mathbf{k}}|=|\mathrm{supp}\,{\overline{\mathbf{k}}}|=1$%
, then it follows from (30) that $M(\ell,n_{1})=\frac{1}{2}%
\left(n_{1}^{2}-n_{1}^{2}\right)+1=1$, and from (32) that $M(\ell,n_{2})=%
\frac{1}{2}\left(n_{2}^{2}-n_{2}^{2}\right)+1=1$. Since $\ell$, $%
n_{1}+n_{2}\geqslant 2$, it is clear that we can choose ${\mathbf{k}}%
^{\ast}=(k_{1}^{\ast},k_{2}^{\ast},\ldots,k_{\ell}^{\ast})\in\mathbb{N}%
_{0}^{\ell}$ such that $|\mathrm{supp}\,{\mathbf{k}}^{\ast}|\geqslant 2$ and 
$\sum_{i=1}^{\ell}k_{i}^{\ast}=n_{1}+n_{2}$. Then

\begin{align*}
M(\ell,n_{1}+n_{2}) &\geqslant {\textstyle\frac{1}{2}%
\left((n_{1}+n_{2})^{2}-|{\mathbf{k}}^{\ast}|^{2}\right)+1} \\[0.5ex]
&= \frac{1}{2}\left(\left(\sum_{i=1}^{\ell}k_{i}^{\ast}\right)^{2}-%
\sum_{i=1}^{\ell}(k_{i}^{\ast})^{2}\right)+1 \\[0.5ex]
&= \sum_{i,j=1,\,i<j}^{\ell}k_{i}^{\ast}k_{j}^{\ast}+1 \\[0.5ex]
&\geqslant 2=M(\ell,n_{1})+M(\ell,n_{2}),
\end{align*}

\noindent as required.

Now suppose $|\mathrm{supp}\,{\mathbf{k}}|\geqslant 2$ or $|\mathrm{supp}\,{%
\overline{\mathbf{k}}}|\geqslant 2$.

Put ${\overline{\overline{\mathbf{k}}}}=(\bar{\bar{k}}_{1},\bar{\bar{k}}%
_{2},\ldots,\bar{\bar{k}}_{\ell})={\mathbf{k}}+{\overline{\mathbf{k}}}$. By
(30) and (32)

\begin{align}
\sum_{i=1}^{\ell}\bar{\bar{k}}_{i}=n_{1}+n_{2}.
\end{align}

\noindent Then

\begin{align}
M(\ell,n_{1}+n_{2}) &\geqslant 
\parbox[t]{22em}{
$\frac{1}{2}\left((n_{1}+n_{2})^{2}-|{\overline{\overline{\mathbf k}}}|^{2}\right)+1$\hspace*{0.6em}[by (34) and the defini-\\[0.2\baselineskip]
 tion of $M(\ell,n_{1}+n_{2})$]}  \notag \\[0.5ex]
& = {\textstyle\frac{1}{2}}\left(n_{1}^{2}+n_{2}^{2}-|{\mathbf{k}}|^{2}-|{%
\overline{\mathbf{k}}}|^{2}\right)+ n_{1}n_{2}-\sum_{i=1}^{\ell}k_{i}\bar{k}%
_{i}+1  \notag \\[1ex]
& = 
\parbox[t]{22em}{
$M(\ell,n_{1})+M(\ell,n_{2})+n_{1}n_{2}-{\displaystyle\sum_{i=1}^{\ell}}k_{i}\bar{k}_{i}-1$\hspace*{0.6em}[by (31)\\
and (33)]
}  \notag \\[0.5ex]
& =
M(\ell,n_{1})+M(\ell,n_{2})+\left(\sum_{i=1}^{\ell}k_{i}\right)\left(%
\sum_{i=1}^{\ell}\bar{k}_{i}\right)-\sum_{i=1}^{\ell}k_{i}\bar{k}_{i}-1 
\notag \\
&\hspace*{1.3em}{[}\parbox[t]{15em}{by (30) and (32)]}  \notag \\[0.5ex]
& = M(\ell,n_{1})+M(\ell,n_{2})+\sum_{i,j=1,\,i\neq j}^{\ell}k_{i}\bar{k}%
_{j}-1.
\end{align}

\noindent Since, by hypothesis, $|\mathrm{supp}\,{\mathbf{k}}|\geqslant 2$
or $|\mathrm{supp}\,{\overline{\mathbf{k}}}|\geqslant 2$, we must have $%
\sum_{i,j=1,\,i\neq j}^{\ell}k_{i}\bar{k}_{j}\geqslant 1$, hence by (35), $%
M(\ell,n_{1}+n_{2})\geqslant M(\ell,n_{1})+M(\ell,n_{2})$, as required.
\end{proof}


\begin{proposition}
The function $M(\ell,n)$ is increasing in both its arguments.
\end{proposition}

\begin{proof}
That $M(\ell,n)$ is increasing in its second argument is an immediate
consequence of Proposition 28.

To show that $M(\ell,n)$ is increasing in its first argument, it suffices to
show that $M(\ell,n)\leqslant M(\ell+1,n)$. Choose ${\mathbf{k}}%
=(k_{1},k_{2},\ldots,k_{\ell})\in\mathbb{N}_{0}^{\ell}$ such that $%
\sum_{i=1}^{\ell}k_{i}=n$ and $M(\ell,n)=\frac{1}{2}\left(n^{2}-|{\mathbf{k}}%
|^{2}\right)+1$. Putting ${\mathbf{k}}^{\prime}=(k_{1},k_{2},\ldots,k_{%
\ell},0)\in\mathbb{N}_{0}^{\ell+1}$, we see that $M(\ell,n)=\frac{1}{2}%
\left(n^{2}-|{\mathbf{k}}|^{2}\right)+1=\frac{1}{2}\left(n^{2}-|{\mathbf{k}}%
^{\prime}|^{2}\right)+1\leqslant M(\ell+1,n)$, as required.
\end{proof}

We attempt now an explicit description of the function $M(\ell,n)$. This is
achieved in Theorem 31. If $\ell$ and $n$ are positive integers with $%
\ell\geqslant n$, then Corollary 27 exhibits the simple formula $M(\ell,n)=%
\frac{1}{2}\left(n^{2}-n\right)+1$. We shall therefore restrict our
attention to the case $\ell\leqslant n$. For such integers $\ell $ and $n$
we denote by $n\,(\mathrm{mod}\,\ell)$ the nonnegative remainder on dividing 
$n$ by $\ell$, that is, the unique integer $r<\ell$ that satisfies

\begin{equation*}
n=\left\lfloor\frac{n}{\ell}\right\rfloor\ell +r.
\end{equation*}

\noindent Let $r=n\,(\mathrm{mod}\,\ell)$ and define ${\mathbf{d}}%
=(d_{1},d_{2},\ldots,d_{\ell})\in\mathbb{N}_{0}^{\ell}$ by

\begin{align}
d_{i}\overset{\mathrm{def}}{=} 
\begin{cases}
\left\lfloor\frac{n}{\ell}\right\rfloor,\hspace*{0.4em}%
\makebox{\rm for
$1\leqslant i\leqslant \ell-r$} \\[2ex] 
\left\lfloor\frac{n}{\ell}\right\rfloor+1,\hspace*{0.4em}%
\makebox{\rm for
$\ell-r<i\leqslant\ell$}.%
\end{cases}%
\end{align}

We omit the proof of the following routine lemma.


\begin{lemma}
Let $\ell$ and $n$ be positive integers with $\ell\leqslant n$ and $r=n\,(%
\mathrm{mod}\,\ell)$. If ${\mathbf{d}}$ is defined as in (36), then

\begin{equation*}
|{\mathbf{d}}|^{2}=(\ell-r)\left\lfloor\frac{n}{\ell}\right\rfloor^2+r\left(%
\left\lfloor\frac{n}{\ell}\right\rfloor+1\right)^2= \frac{n^{2}-r^{2}}{\ell}%
+r.
\end{equation*}
\end{lemma}


\begin{theorem}
Let $\ell$ and $n$ be positive integers with $\ell\leqslant n$ and $r=n\,(%
\mathrm{mod}\,\ell)$. If ${\mathbf{d}}$ is defined as in (36), then

\begin{align*}
M(\ell,n) & ={\textstyle\frac{1}{2}}\left(n^{2}-|{\mathbf{d}}|^{2}\right)+1
\\[1ex]
& =\frac{1}{2}\left(n^2 - (\ell-r)\left\lfloor\frac{n}{\ell}\right\rfloor^2
- r\left(\left\lfloor\frac{n}{\ell}\right\rfloor + 1\right)^2\right) + 1 \\%
[1ex]
& = \frac{n^{2}(\ell-1)}{2\ell}+\frac{1}{2}\left(\frac{r^{2}}{\ell}%
-r\right)+1.
\end{align*}
\end{theorem}

\begin{proof}
It is clear from the definition of ${\mathbf{d}}$ in (36) that $%
\sum_{i=1}^{\ell}d_{i}=n$ and $|d_{i}-d_{j}|\leqslant 1$ for all $%
i,j\in\{1,2,\ldots,\ell\}$. Hence, by Proposition 26 ((b)$\Rightarrow$(a)),

\begin{align*}
M(\ell,n) &= {\textstyle\frac{1}{2}}\left(n^{2}-|{\mathbf{d}}|^{2}\right)+1 
\notag \\[1ex]
& = \frac{1}{2}\left(n^2 - (\ell-r)\left\lfloor\frac{n}{\ell}\right\rfloor^2
- r\left(\left\lfloor\frac{n}{\ell}\right\rfloor + 1\right)^2\right) + 1%
\hspace*{0.6em}\makebox{\rm [by Lemma 30]} \\[1ex]
& = \frac{1}{2}\left(n^{2}-\left(\frac{n^{2}-r^{2}}{\ell}+r\right)\right)+1%
\hspace*{0.6em}\makebox{\rm [by Lemma 30]} \\[1ex]
& = \frac{1}{2}\left(n^{2}-\frac{n^{2}}{\ell}+\frac{r^{2}}{\ell}-r\right)+1
\\[1ex]
& = \frac{n^{2}(\ell-1)}{2\ell}+\frac{1}{2}\left(\frac{r^{2}}{\ell}%
-r\right)+1.
\end{align*}
\end{proof}

Suppose $F$ is any field and $R$ the algebra of $n\times n$ matrices over $F$
of type

\begin{equation*}
{\mathbf{d}}=(\overbrace{ {\textstyle\left\lfloor\frac{n}{\ell}\right\rfloor}%
, {\textstyle\left\lfloor\frac{n}{\ell}\right\rfloor},\ldots,{\textstyle%
\left\lfloor\frac{n}{\ell}\right\rfloor}}^{(\ell-r)\mathrm{\ times}},%
\overbrace{{\textstyle\left\lfloor\frac{n}{\ell}\right\rfloor+1},{\textstyle%
\left\lfloor\frac{n}{\ell}\right\rfloor+1},\ldots, {\textstyle\left\lfloor%
\frac{n}{\ell}\right\rfloor +1}}^{r\mathrm{\ times}}),
\end{equation*}

\noindent with $n\geqslant\ell\geqslant 2$.

Figure 1 is a pictorial representation of the radical $J$ of $R$.

\begin{figure}[!ht]
\caption{Pictorial representation of the radical $J$ of $R$}\vspace*{1ex}
\par
\begin{center}
\begin{tikzpicture}

\draw [dotted](0,9)--(9,0);

\draw {[line width=0.1pt](0,0) rectangle (9,9)};

\draw [fill=gray!30][line width=0.1pt](9,9) rectangle (0.5,8.5);
\draw [fill=gray!30][line width=0.1pt](9,8.5) rectangle (1,8);
\draw [fill=gray!30][line width=0.1pt](9,7) rectangle (2.5,6.5);
\draw [fill=gray!30][line width=0.1pt](9,6.5) rectangle (3.5,5.5);
\draw [fill=gray!30][line width=0.1pt](9,5.5) rectangle (4.5,4.5);
\draw [fill=gray!30][line width=0.1pt](9,2) rectangle (8,1);

\node [right] at (8.5,7.25) {$\cdot$};
\node [right] at (8.5,7.5) {$\cdot$};
\node [right] at (8.5,7.75) {$\cdot$};

\node [right] at (8.5,2.25) {$\cdot$};
\node [right] at (8.5,2.5) {$\cdot$};
\node [right] at (8.5,2.75) {$\cdot$};
\node [right] at (8.5,3) {$\cdot$};
\node [right] at (8.5,3.25) {$\cdot$};
\node [right] at (8.5,3.5) {$\cdot$};
\node [right] at (8.5,3.75) {$\cdot$};
\node [right] at (8.5,4) {$\cdot$};
\node [right] at (8.5,4.25) {$\cdot$};

\node [right] at (9.32,7.25) {$\cdot$};
\node [right] at (9.32,7.5) {$\cdot$};
\node [right] at (9.32,7.75) {$\cdot$};

\node [right] at (9.32,2) {$\cdot$};
\node [right] at (9.32,2.25) {$\cdot$};
\node [right] at (9.32,2.5) {$\cdot$};
\node [right] at (9.32,2.75) {$\cdot$};
\node [right] at (9.32,3) {$\cdot$};
\node [right] at (9.32,3.25) {$\cdot$};
\node [right] at (9.32,3.5) {$\cdot$};
\node [right] at (9.32,3.75) {$\cdot$};
\node [right] at (9.32,4) {$\cdot$};
\node [right] at (9.32,4.25) {$\cdot$};
\node [right] at (9.32,4.5) {$\cdot$};

\node [left] at (0,4.5) {$J\,=\,$};


\draw [decorate,decoration={brace,amplitude=2.2pt,mirror,raise=4pt},yshift=0pt]
(9,8.55) -- (9,8.95) node [black,midway,xshift=1.1cm]{\footnotesize $d_{1}=\left\lfloor\frac{n}{\ell}\right\rfloor$};

\draw [decorate,decoration={brace,amplitude=2.2pt,mirror,raise=4pt},yshift=0pt]
(9,8.05) -- (9,8.45) node [black,midway,xshift=1.1cm]{\footnotesize $d_{2}=\left\lfloor\frac{n}{\ell}\right\rfloor$};

\draw [decorate,decoration={brace,amplitude=2.2pt,mirror,raise=4pt},yshift=0pt]
(9,6.55) -- (9,6.95) node [black,midway,xshift=1.3cm]{\footnotesize $d_{\ell-r}=\left\lfloor\frac{n}{\ell}\right\rfloor$};

\draw [decorate,decoration={brace,amplitude=3.1pt,mirror,raise=4pt},yshift=0pt]
(9,5.55) -- (9,6.45) node [black,midway,xshift=1.8cm]{\footnotesize $d_{\ell-r+1}=\left\lfloor\frac{n}{\ell}\right\rfloor+1$};

\draw [decorate,decoration={brace,amplitude=3.1pt,mirror,raise=4pt},yshift=0pt]
(9,4.55) -- (9,5.45) node [black,midway,xshift=1.8cm]{\footnotesize $d_{\ell-r+2}=\left\lfloor\frac{n}{\ell}\right\rfloor+1$};

\draw [decorate,decoration={brace,amplitude=3.1pt,mirror,raise=4pt},yshift=0pt]
(9,1.05) -- (9,1.95) node [black,midway,xshift=1.6cm]{\footnotesize $d_{\ell-1}=\left\lfloor\frac{n}{\ell}\right\rfloor+1$};

\draw [decorate,decoration={brace,amplitude=3.1pt,mirror,raise=4pt},yshift=0pt]
(9,0.05) -- (9,0.95) node [black,midway,xshift=1.4cm]{\footnotesize $d_{\ell}=\left\lfloor\frac{n}{\ell}\right\rfloor+1$};

\end{tikzpicture}
\end{center}
\end{figure}
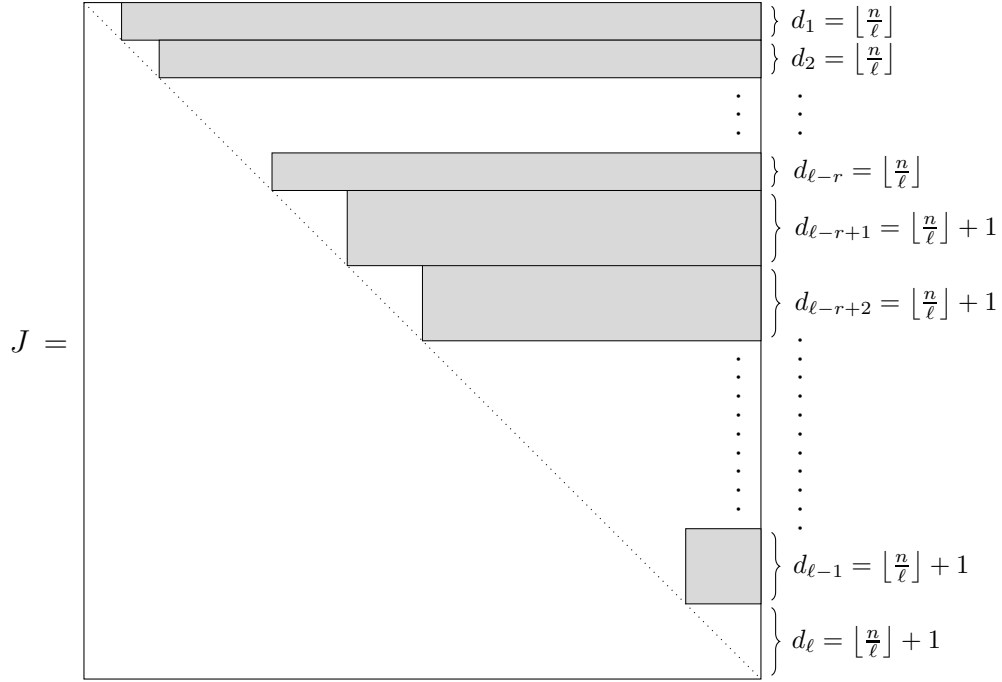

\noindent Inasmuch as $R$ has the form $R=FI_{n}+J$ with $J$ satisfying $%
J^{\ell}=0$, it follows that $R$ is Lie nilpotent of index $\ell-1$. (This
assertion is explained in more detail in the discussion following the
statement of the Conjecture (7).) Moreover,

\begin{align*}
\mathrm{dim}_{F}R &=
\sum_{j=1}^{\ell-1}d_{j}\left(n-\sum_{i=1}^{j}d_{i}\right)+1 \\[0.5ex]
&= {\textstyle\frac{1}{2}}\left(n^{2}-|{\mathbf{d}}|^{2}\right)+1 \\[0.5ex]
&= M(\ell,n)\hspace*{0.6em}\makebox{\rm [by Theorem 31].}
\end{align*}

\noindent Thus $R$ is an $F$-subalgebra of $\mathbb{M}_{n}(F)$ whose
dimension is maximal amongst $F$-subalgebras of $\mathbb{M}_{n}(F)$ with Lie
nilpotence index $\ell-1$.

If $\frac{1}{2}\left(n^{2}-|{\mathbf{k}}|^{2}\right)+1$ is interpreted as a
real-valued function of real variables ${\mathbf{k}}=(k_{1},k_{2},\ldots,k_{%
\ell})\in{\mathbb{R}}^{\ell}$, the methods of multivariable Calculus show
that the function $\frac{1}{2}\left(n^{2}-|{\mathbf{k}}|^{2}\right)+1$,
subject to the constraint $\sum_{i=1}^{\ell}k_{i}=n$, attains a maximum of $%
\frac{n^{2}(\ell-1)}{2\ell}+1$ at ${\mathbf{k}}=(\frac{n}{\ell},\frac{n}{\ell%
},\ldots,\frac{n}{\ell})\in{\mathbb{R}}^{\ell}$. Thus

\begin{align}
\left\lfloor\frac{n^{2}(\ell-1)}{2\ell}\right\rfloor+1\geq M(\ell,n).
\end{align}

We explore now instances in which (37) is an equation, a situation that
arises precisely when $D<1$, where

\begin{equation*}
D\,\overset{\mathrm{def}}{=}\frac{n^{2}(\ell-1)}{2\ell}+1-M(\ell,n).
\end{equation*}

\noindent It follows from Theorem 31 that

\begin{align}
D=\frac{1}{2}\left(r-\frac{r^{2}}{\ell}\right)
\end{align}

\noindent where $r=n\,(\mathrm{mod}\,\ell)$. Observe that $D=D(r,\ell)$ is a
function only of $r$ and $\ell$.

Figure 2 is a sketch of the level curve $D(r,\ell)=1$ in the $r\ell$-plane,
interpreting $r$ and $\ell$ as real-valued variables. A simple calculation
shows that the curve has equation

\begin{equation*}
\ell=\frac{r^{2}}{r-2}.
\end{equation*}

\noindent Its essential features are obtained using elementary Calculus.

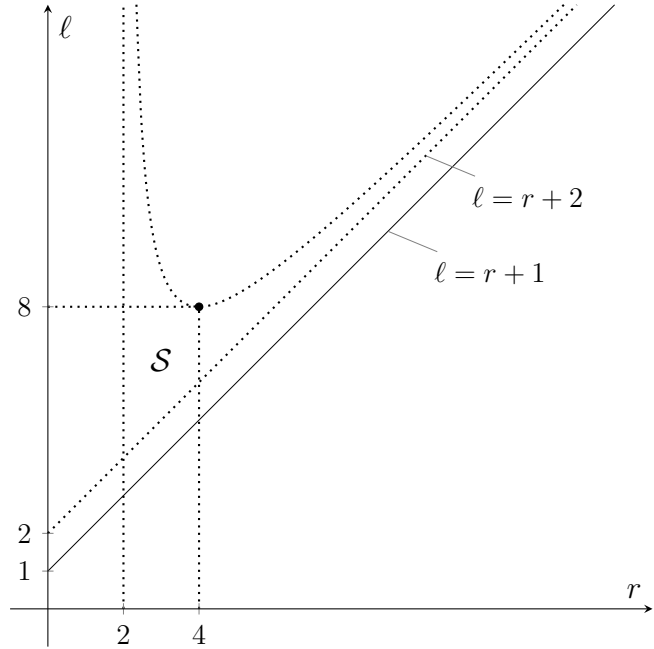
\begin{figure}[!ht]
\caption{The level curve $D(r,\ell)=1$}\vspace*{3ex}
\par
\begin{tikzpicture}[scale=1]
\begin{axis}[
ytick={1,2,8},
xtick={2,4},
yticklabels={$\scalebox{0.9}{1}$,$\scalebox{0.9}{2}$,$\scalebox{0.9}{8}$},
xticklabels={$\scalebox{0.9}{2}$,$\scalebox{0.9}{4}$},
    axis x line=middle,
    axis y line=middle,
    axis on top,
    smooth,
    xlabel=$r$,
    ylabel=$\ell$,
    restrict y to domain=0:16,
    xmin=-1, xmax=16,
    ymin=-1, ymax=16,
    width=0.8\textwidth,
    height=0.8\textwidth,
    legend style={at={(0.02,0.97)},anchor=north west},
]

\path[name path=topaxis] (axis cs:0,16) -- (axis cs:16,16);
\path[name path=bottomaxis] (axis cs:0,0) -- (axis cs:16,0);

\addplot[name path=asymptoteattwo, thick, style=dotted] coordinates {(2,0) (2,16)};
\addplot[name path=shorthorizontal, thick, style=dotted] coordinates {(0,8) (4,8)};
\addplot[name path=shortvertical, thick, style=dotted] coordinates {(4,0) (4,8)};

\path[name path=linesegment] (axis cs:2.34314575,16)--(axis cs:13.6568542,16);
\addplot[name path=shortcurve, draw=black, domain=2.34314575:13.6568542, samples=100, thick, style=dotted]{(x^2)/(x-2)};

\addplot[name path=obliqueline,draw=black, line width=0.3pt, domain=0:16, samples=400]{x+1};
\addplot[name path=topobliqueline,draw=black, thick, domain=0:16, samples=400, style=dotted]{x+2};

\addplot[color=blue, fill=blue, fill opacity=0.1] fill between [of=topaxis and obliqueline];

\addplot[color=white, fill=white, fill opacity=1] fill between [of=linesegment and shortcurve];


\draw(axis cs:4,8) [fill] circle [blue, radius=1];

\node[black,above] at (axis cs:3,6){${\mathcal S}$};

\node[coordinate,pin=-30:{\scalebox{0.9}{$\ell=r+1$}}] at (axis cs:9,10){};

\node[coordinate,pin=-30:{\scalebox{0.9}{$\ell=r+2$}}] at (axis cs:10,12){};

\end{axis}

\end{tikzpicture}
\end{figure}

\noindent The shaded region is

\begin{equation*}
{\mathcal{S}}\,\overset{\mathrm{def}}{=}\{(r,\ell)\in{\mathbb{R}}%
^{2}:0\leqslant r\leqslant\ell-1\makebox{ {\rm and} }D(r,\ell)<1\}.
\end{equation*}

The content of Theorem 32 below is easily gleaned from Figure 2 by
assembling together points $(r,\ell)$ belonging to ${\mathcal{S}}$ that have
integral coordinates.


\begin{theorem}
Let $\ell$ and $n$ be positive integers with $\ell\leqslant n$ and $r=n\,(%
\mathrm{mod}\,\ell)$. Then the following statements are equivalent:

\begin{list}{{(\alph{standard})}\hfill}{\usecounter{standard}\addtocounter{standard}{0}\setlength{\topsep}{1ex}\setlength{\labelwidth}{1.5\parindent}\setlength{\labelsep}{0.2\parindent}\setlength{\leftmargin}{2.3\parindent}\setlength{\listparindent}{0pt}\setlength{\itemsep}{5pt}}

\item ${\displaystyle M(\ell,n)=\left\lfloor\frac{n^{2}(\ell-1)}{2\ell}\right\rfloor+1}$;\\[0ex]

\item $(r,\ell)$ belongs to one of the following (disjoint) sets:

\begin{list}%
{{(\roman{substandard})}\hfill}{\usecounter{substandard}%
\addtocounter{substandard}{0}%
\setlength{\topsep}{1ex}%
\setlength{\labelwidth}{2em}%
\setlength{\labelsep}{0.2\parindent}%
\setlength{\leftmargin}{2em}%
\setlength{\listparindent}{0pt}%
\setlength{\itemsep}{6pt}%
}

\item $\{(r,\ell):0\leqslant r\leqslant \ell-1$ and $1\leqslant\ell\leqslant 7\}$;

\item $\{(r,\ell):0\leqslant r\leqslant 2$ and $\ell\geqslant 8\}$;

\item $\{(r,r+1):r\geqslant 7\}\cup\{(r,r+2):r\geqslant 7\}$;

\item $\{(3,8),(5,8)\}$.
\end{list}
\end{list}
\end{theorem}


\begin{remark}
The reader will observe with reference to Theorem 32(b)(i), that if, amongst
others, $1\leqslant\ell\leqslant 7$, we have the simplified formula

\begin{equation*}
M(\ell,n)=\left\lfloor\frac{n^{2}(\ell-1)}{2\ell}\right\rfloor+1.
\end{equation*}

\noindent In particular, if $\ell=2$, then

\begin{equation*}
M(2,n)=\left\lfloor\frac{n^{2}}{4}\right\rfloor+1,
\end{equation*}

\noindent which corresponds with the upper bound in Schur's classical result.
\end{remark}


\section{An illustrative example}

The main body of theory developed in Section 5 is based on the triple of
sequences $\{R_{k}\}_{k\in\mathbb{N}}$, $\{J_{k}\}_{k\in\mathbb{N}}$ and $%
\{U_{k}\}_{k\in\mathbb{N}}$ defined in (14). In this section we show that
the terms in these sequences are easily visualized in the case where $R$ is
the algebra of $n\times n$ matrices over field $F$ of type $%
(d_{1},d_{2},\ldots,d_{\ell})$. Indeed, this special case provides the germ
for our proof strategy.

Let $F$ be any field and $(d_{1},d_{2},\ldots,d_{\ell})$ any sequence of
positive integers satisfying $\sum_{i=1}^{\ell}d_{i}=n$ with $\ell\geqslant
2 $. Let $R$ be the algebra of $n\times n$ matrices over $F$ of type $%
(d_{1},d_{2},\ldots,d_{\ell})$. We saw in Section 1 that the radical $J$ of $%
R$ has pictorial representation

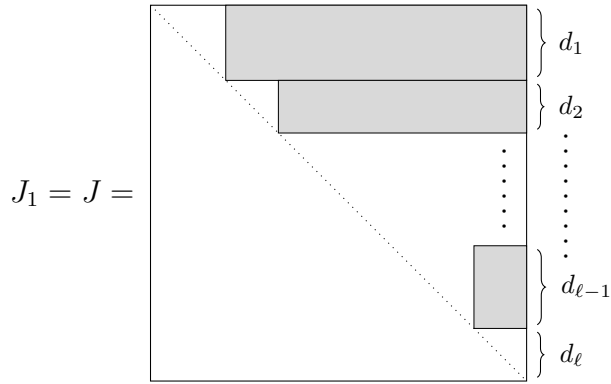
\vspace*{1\baselineskip}

\begin{center}
\begin{tikzpicture}

\draw [dotted](0,5)--(5,0);

\draw {[line width=0.1pt](0,0) rectangle (5,5)};

\draw [fill=gray!30][line width=0.1pt](5,5) rectangle (1,4);
\draw [fill=gray!30][line width=0.1pt](5,4) rectangle (1.7,3.3);
\draw [fill=gray!30][line width=0.1pt](5,1.8) rectangle (4.3,0.7);

\node [right] at (4.5,2.05) {$\cdot$};
\node [right] at (4.5,2.25) {$\cdot$};
\node [right] at (4.5,2.45) {$\cdot$};
\node [right] at (4.5,2.65) {$\cdot$};
\node [right] at (4.5,2.85) {$\cdot$};
\node [right] at (4.5,3.05) {$\cdot$};

\node [right] at (5.32,1.65) {$\cdot$};
\node [right] at (5.32,1.85) {$\cdot$};
\node [right] at (5.32,2.05) {$\cdot$};
\node [right] at (5.32,2.25) {$\cdot$};
\node [right] at (5.32,2.45) {$\cdot$};
\node [right] at (5.32,2.65) {$\cdot$};
\node [right] at (5.32,2.85) {$\cdot$};
\node [right] at (5.32,3.05) {$\cdot$};
\node [right] at (5.32,3.25) {$\cdot$};

\node [left] at (0,2.5) {$J_{1}=J=\,$};

\draw [decorate,decoration={brace,amplitude=3pt,mirror,raise=4pt},yshift=0pt]
(5,4.05) -- (5,4.95) node [black,midway,xshift=0.6cm]{\footnotesize $d_{1}$};

\draw [decorate,decoration={brace,amplitude=2.5pt,mirror,raise=4pt},yshift=0pt]
(5,3.35) -- (5,3.95) node [black,midway,xshift=0.6cm]{\footnotesize $d_{2}$};

\draw [decorate,decoration={brace,amplitude=3.1pt,mirror,raise=4pt},yshift=0pt]
(5,0.75) -- (5,1.75) node [black,midway,xshift=0.8cm]{\footnotesize $d_{\ell-1}$};

\draw [decorate,decoration={brace,amplitude=2.5pt,mirror,raise=4pt},yshift=0pt]
(5,0.05) -- (5,0.65) node [black,midway,xshift=0.6cm]{\footnotesize $d_{\ell}$};

\end{tikzpicture}
\end{center}

\vspace*{1\baselineskip}

\noindent Observe that $\mathrm{dim}_{F}J_{1}$ corresponds with the sum of
the dimensions (to be visualized as areas) of each of the $\ell -1$ blocks
that make up $J_{1}$. With this perspective we see that

\begin{equation*}
\mathrm{dim}_{F}J_{1}=\overbrace{d_{1}(n-d_{1})}^{\mathrm{1st \ block}} + 
\overbrace{d_{2}(n-d_{1}-d_{2})}^{\mathrm{2nd \ block}}+ \cdots+\overbrace{%
d_{\ell-1}(n-d_{1}-\cdots-d_{\ell-1})}^{(\ell-1)\mathrm{th \ block}}.
\end{equation*}

\noindent Note also that

\begin{equation*}
J_{1}^{\ell}=0,
\end{equation*}

\noindent from which we infer that $R_{1}$ is Lie nilpotent of index $\ell-1$%
. (This inference is explained in the discussion following the statement of
the Conjecture (7).)

Take $V=F^{n}$, which in this context is to be visualized as a $1\times n$
block thus

\vspace*{1\baselineskip}

\begin{center}
\begin{tikzpicture}
\node [left] at (0,0) {$V=\,$};
\draw [fill=gray!30] (0,-0.15) rectangle (5,0.15);
\end{tikzpicture}
\end{center}

\vspace*{1\baselineskip}

\noindent Given the above pictorial representations of $V$ and $J_{1}$, we
see that

\vspace*{1\baselineskip}

\begin{center}
\begin{tikzpicture}
\node [left] at (0,0) {$VJ_{1}\,=\,$};
\draw (0,-0.15) rectangle (1,0.15);
\draw [fill=gray!30] (1,-0.15) rectangle (5,0.15);
\draw [decorate,decoration={brace,amplitude=4pt,raise=-2pt},yshift=10pt]
(0.05,0) -- (0.95,0) node [black,above,midway,xshift=0cm]{\makebox[0.4em][l]{\footnotesize $d_{1}$\,(zero entries)}};
\end{tikzpicture}
\end{center}

\vspace*{1\baselineskip}

\noindent Choosing

\begin{center}
\begin{tikzpicture}
\node [left] at (0,0) {$U_{1}\,=\,$};
\draw [fill=gray!30] (0,-0.15) rectangle (1,0.15);
\draw (1,-0.15) rectangle (5,0.15);
\draw [decorate,decoration={brace,amplitude=4pt,raise=-2pt},yshift=10pt]
(0.05,0) -- (0.95,0) node [black,above,midway,xshift=0cm]{\makebox[0.4em][l]{\footnotesize $d_{1}$}};
\end{tikzpicture}
\end{center}

\vspace*{1\baselineskip}

\noindent we see that

\vspace*{1\baselineskip}

\begin{center}
\begin{tikzpicture}

\draw [dotted](0,5)--(5,0);

\draw {[line width=0.1pt](0,0) rectangle (5,5)};

\draw [fill=gray!30][line width=0.1pt](5,4) rectangle (1.7,3.3);
\draw [fill=gray!30][line width=0.1pt](5,1.8) rectangle (4.3,0.7);

\node [right] at (4.5,2.05) {$\cdot$};
\node [right] at (4.5,2.25) {$\cdot$};
\node [right] at (4.5,2.45) {$\cdot$};
\node [right] at (4.5,2.65) {$\cdot$};
\node [right] at (4.5,2.85) {$\cdot$};
\node [right] at (4.5,3.05) {$\cdot$};

\node [right] at (5.32,1.65) {$\cdot$};
\node [right] at (5.32,1.85) {$\cdot$};
\node [right] at (5.32,2.05) {$\cdot$};
\node [right] at (5.32,2.25) {$\cdot$};
\node [right] at (5.32,2.45) {$\cdot$};
\node [right] at (5.32,2.65) {$\cdot$};
\node [right] at (5.32,2.85) {$\cdot$};
\node [right] at (5.32,3.05) {$\cdot$};
\node [right] at (5.32,3.25) {$\cdot$};

\node [left] at (0,2.5) {$J_{2}=(0:^{R_{1}}U_{1})=\,$};

\draw [decorate,decoration={brace,amplitude=3pt,mirror,raise=4pt},yshift=0pt]
(5,4.05) -- (5,4.95) node [black,midway,xshift=0.6cm]{\makebox[0.8em][l]{\footnotesize $d_{1}$\,(zero rows)}};

\draw [decorate,decoration={brace,amplitude=2.5pt,mirror,raise=4pt},yshift=0pt]
(5,3.35) -- (5,3.95) node [black,midway,xshift=0.6cm]{\footnotesize $d_{2}$};

\draw [decorate,decoration={brace,amplitude=3.1pt,mirror,raise=4pt},yshift=0pt]
(5,0.75) -- (5,1.75) node [black,midway,xshift=0.8cm]{\footnotesize $d_{\ell-1}$};

\draw [decorate,decoration={brace,amplitude=2.5pt,mirror,raise=4pt},yshift=0pt]
(5,0.05) -- (5,0.65) node [black,midway,xshift=0.6cm]{\footnotesize $d_{\ell}$};

\end{tikzpicture}
\end{center}

\vspace*{1\baselineskip}

\noindent Here:

\begin{list}{{(\alph{standard})}\hfill}{\usecounter{standard}\addtocounter{standard}{0}\setlength{\topsep}{1ex}\setlength{\labelwidth}{1.5\parindent}\setlength{\labelsep}{0.7\parindent}\setlength{\leftmargin}{2.3\parindent}\setlength{\listparindent}{0pt}\setlength{\itemsep}{6pt}}

\item[$\triangleright$] ${\rm dim}_{F}J_{2}=\overbrace{d_{2}(n-d_{1}-d_{2})}^{2{\rm nd \ block}}+\cdots+
\overbrace{d_{\ell-1}(n-d_{1}-\cdots-d_{\ell-1})}^{(\ell-1){\rm th \ block}}$;

\item[$\triangleright$] $J^{\ell-1}_{2}=0$;

\item[$\triangleright$] $R_{2}$ is Lie nilpotent of index $\ell-2$;

\item[$\triangleright$]
\hspace*{-1em}
\raisebox{-1ex}{
\begin{tikzpicture}
\node [left] at (0,0) {$VJ_{1}J_{2}\,=\,$};
\draw (0,-0.15) rectangle (1.7,0.15);
\draw [fill=gray!30] (1.7,-0.15) rectangle (5,0.15);
\draw [decorate,decoration={brace,amplitude=4pt,raise=-2pt},yshift=10pt]
(0.05,0) -- (1.65,0) node [black,above,midway,xshift=0cm]{\footnotesize $d_{1}+d_{2}$};
\end{tikzpicture}
};

\item[$\triangleright$]
\hspace*{-1em}
\raisebox{-1ex}{
\begin{tikzpicture}
\node [left] at (0,0) {$U_{2}\,=\,$};
\draw (0,-0.15) rectangle (1,0.15);
\draw [fill=gray!30] (1,-0.15) rectangle (1.7,0.15);
\draw (1.7,-0.15) rectangle (5,0.15);
\draw [decorate,decoration={brace,amplitude=3pt,raise=-2pt},yshift=10pt]
(0.05,0) -- (0.95,0) node [black,above,midway,xshift=0cm]{\makebox[0.4em][l]{\footnotesize $d_{1}$}};
\draw [decorate,decoration={brace,amplitude=3pt,raise=-2pt},yshift=10pt]
(1.05,0) -- (1.65,0) node [black,above,midway,xshift=0cm]{\makebox[0.4em][l]{\footnotesize $d_{2}$}};
\end{tikzpicture}
}.
\end{list}

\noindent Continuing in this manner, we arrive at a smallest $F$-subalgebra
of $R$ properly containing $FI_{n}$, namely $R_{\ell-1}$, and this has
radical comprising a single block

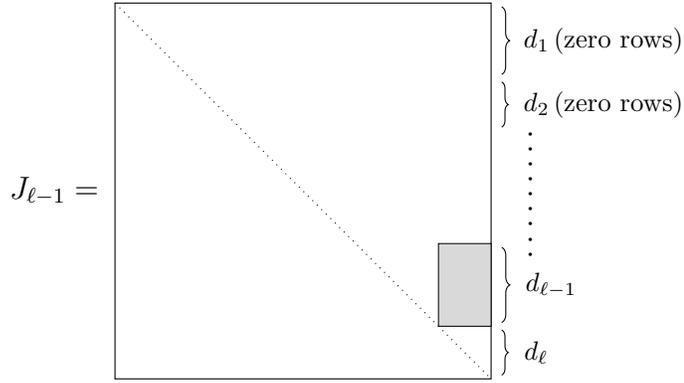
\vspace*{1\baselineskip}

\begin{center}
\begin{tikzpicture}

\draw [dotted](0,5)--(5,0);

\draw {[line width=0.1pt](0,0) rectangle (5,5)};

\draw [fill=gray!30][line width=0.1pt](5,1.8) rectangle (4.3,0.7);


\node [right] at (5.32,1.65) {$\cdot$};
\node [right] at (5.32,1.85) {$\cdot$};
\node [right] at (5.32,2.05) {$\cdot$};
\node [right] at (5.32,2.25) {$\cdot$};
\node [right] at (5.32,2.45) {$\cdot$};
\node [right] at (5.32,2.65) {$\cdot$};
\node [right] at (5.32,2.85) {$\cdot$};
\node [right] at (5.32,3.05) {$\cdot$};
\node [right] at (5.32,3.25) {$\cdot$};

\node [left] at (0,2.5) {$J_{\ell-1}=\,$};

\draw [decorate,decoration={brace,amplitude=3pt,mirror,raise=4pt},yshift=0pt]
(5,4.05) -- (5,4.95) node [black,midway,xshift=0.6cm]{\makebox[0.8em][l]{\footnotesize $d_{1}$\,(zero rows)}};

\draw [decorate,decoration={brace,amplitude=2.5pt,mirror,raise=4pt},yshift=0pt]
(5,3.35) -- (5,3.95) node [black,midway,xshift=0.6cm]{\makebox[0.8em][l]{\footnotesize $d_{2}$\,(zero rows)}};

\draw [decorate,decoration={brace,amplitude=3.1pt,mirror,raise=4pt},yshift=0pt]
(5,0.75) -- (5,1.75) node [black,midway,xshift=0.8cm]{\footnotesize $d_{\ell-1}$};

\draw [decorate,decoration={brace,amplitude=2.5pt,mirror,raise=4pt},yshift=0pt]
(5,0.05) -- (5,0.65) node [black,midway,xshift=0.6cm]{\footnotesize $d_{\ell}$};

\end{tikzpicture}
\end{center}

\vspace*{1\baselineskip}

\noindent Here:

\begin{list}{{(\alph{standard})}\hfill}{\usecounter{standard}\addtocounter{standard}{0}\setlength{\topsep}{1ex}\setlength{\labelwidth}{1.5\parindent}\setlength{\labelsep}{0.7\parindent}\setlength{\leftmargin}{2.3\parindent}\setlength{\listparindent}{0pt}\setlength{\itemsep}{6pt}}

\item[$\triangleright$] ${\rm dim}_{F}J_{\ell-1}=\overbrace{d_{\ell-1}(n-d_{1}-\cdots-d_{\ell-1})}^{(\ell-1){\rm th \ block}}$;

\item[$\triangleright$] $J^{2}_{\ell-1}=0$;

\item[$\triangleright$] $R_{\ell-1}$ is Lie nilpotent of index $1$ and is thus commutative;

\item[$\triangleright$]
\raisebox{-1ex}{
\begin{tikzpicture}
\node [left] at (0,0) {$VJ_{1}J_{2}\ldots J_{\ell-1}\,=\,$};
\draw (0,-0.15) rectangle (4.3,0.15);
\draw [fill=gray!30] (4.3,-0.15) rectangle (5,0.15);
\draw [decorate,decoration={brace,amplitude=5pt,raise=-2pt},yshift=10pt]
(0.05,0) -- (4.25,0) node [black,above,midway,xshift=0cm]{\footnotesize $d_{1}+d_{2}+\cdots+d_{\ell-1}$};
\end{tikzpicture}
};

\item[$\triangleright$]
\raisebox{-1ex}{
\begin{tikzpicture}
\node [left] at (0,0) {$U_{\ell-1}\,=\,$};
\draw (0,-0.15) rectangle (3.2,0.15);
\draw [fill=gray!30] (3.2,-0.15) rectangle (4.3,0.15);
\draw (4.3,-0.15) rectangle (5,0.15);
\draw [decorate,decoration={brace,amplitude=3pt,raise=-2pt},yshift=10pt]
(3.25,0) -- (4.25,0) node [black,above,midway,xshift=0cm]{\footnotesize $d_{\ell-1}$};
\draw [decorate,decoration={brace,amplitude=3pt,raise=-2pt},yshift=10pt]
(4.35,0) -- (4.95,0) node [black,above,midway,xshift=0cm]{\footnotesize $d_{\ell}$};
\end{tikzpicture}
}.
\end{list}


\section{Open questions}

\begin{list}{{(\arabic{standard})}\hfill}{\usecounter{standard}\addtocounter{standard}{0}\setlength{\topsep}{1ex}\setlength{\labelwidth}{1.5\parindent}\setlength{\labelsep}{0.2\parindent}\setlength{\leftmargin}{2.3\parindent}\setlength{\listparindent}{0pt}\setlength{\itemsep}{5pt}}

\item The sequence $\{U_k\}_{k \in \N}$ of $F$-subspace complements defined in (14) is not unique. This has the consequence that the sequence
$\{R_{k}\}_{k\in\N}$ of $F$-subalgebras of $R$ is not uniquely determined by $R$. Are the $R_{k}$ unique to within isomorphism perhaps? Or failing this, are the dimensions (over $F$) of the $R_{k}$ unique?

\item Recall that if ${\mathfrak g}$ is a Lie algebra, then the {\em Derived Series} $\{{\mathfrak g}^{[m]}\}_{m\in\N}$ for ${\mathfrak g}$ is defined
recursively as follows

\begin{align*}
{\mathfrak g}^{[1]} &\define {\mathfrak g},\hspace*{0.3em}\makebox{\rm and}\\[0.5ex]
{\mathfrak g}^{[m]} &\define [{\mathfrak g}^{[m-1]},{\mathfrak g}^{[m-1]}],\hspace*{0.3em}\makebox{\rm for $m>1$.}
\end{align*}

\noindent We say ${\mathfrak g}$ is {\em solvable} if ${\mathfrak g}^{[m]}=0$ for some $m\in\N$, $m>1$, and more specifically,
{\em solvable of index $m$}, if ${\mathfrak g}^{[m+1]}=0$. We call a ring $R$ {\em Lie solvable} [resp.~{\em Lie solvable of index $m$}] if $R$, considered as a Lie algebra via the commutator, is solvable [resp.~solvable of index $m$]. If $\{{\mathfrak g}_{[m]}\}_{m\in\N}$ denotes the Lower Central Series for ${\mathfrak g}$, it is easily seen that ${\mathfrak g}^{[m]}\subseteq{\mathfrak g}_{[m]}$ for all $m\in\N$, from which it follows that every ring $R$ that is Lie nilpotent of index $m$, is also Lie solvable of index $m$. This being so, it is natural to ask whether the main theorems of this paper remain valid if the condition `Lie nilpotent of index $m$' is substituted with the weaker `Lie solvable of index $m$'.

\item Expressed in terms that make no explicit reference to the overlying matrix ring, a key result in this paper asserts that if $R$ is an $F$-algebra
with Lie nilpotence index $m$, and $V$ is any faithful right $R$-module, then ${\rm dim}_{F}R\leqslant M(m+1,{\rm dim}_{F}V)$. (This is Theorem 24 with ${\rm dim}_{F}V$ in place of $n$.) We ask whether the same inequality holds if the requirement that $R$ be a finite dimensional $F$-algebra is weakened to $R$ being merely a (two-sided) artinian ring. In such a situation, `$R$-module length' takes the place of `$F$-dimension' thus yielding the conjecture\medskip

\begin{quote}
{\em If $R$ is a (two-sided) artinian ring with Lie nilpotence index $m$ and $V$ is any faithful right $R$-module with finite composition length, then

\[
{\rm length}\,R_{R}\leqslant M(m+1,{\rm length}\,V_{R}).
\]
}
\end{quote}

\noindent In the case where $m=1$, the above reduces to the question \cite[Section 5, Open problem (a), p.~562]{Gu76} that is answered in \cite{Co93}\footnote{Although the ring $R$ is assumed to be local in Gustafson's formulation, Cowsik's proof does not assume localness.}.

\end{list}

\end{document}